\documentclass[11pt,a4paper,reqno]{amsart}
\usepackage{vmargin,color}
\usepackage[latin1]{inputenc}
\usepackage{amsmath,amsfonts,amsthm,epsfig,graphicx}
\usepackage{tikz}
\usepackage{tikz-3dplot}
\usepackage{pgfplots}
\usetikzlibrary{intersections, pgfplots.fillbetween}
\usepackage{pstricks,pst-plot,multido}
\usepackage{caption}
\usepackage{enumerate}
\usepackage{esint} 
\usepackage{pgfplots}
\usepackage{graphicx,color,calc}
\usetikzlibrary{decorations.pathreplacing}
\usepackage{comment}
\pgfplotsset{compat=1.10}
\usepgfplotslibrary{fillbetween}
\usepgfplotslibrary{polar}
\usetikzlibrary{patterns}

\usepackage{subfigure}

\definecolor{light-gray}{gray}{0.6}
\definecolor{really-light-gray}{gray}{0.8}

\newcommand{\mres}{\mathbin{\vrule height 1.6ex depth 0pt width  
0.13ex\vrule height 0.13ex depth 0pt width 1.3ex}}    

\def\big{\bigskip}

\newtheorem{theorem}{Theorem}[section]

\newtheorem{proposition}[theorem]{Proposition}
\newtheorem{lemma}[theorem]{Lemma}
\newtheorem{corollary}[theorem]{Corollary}

\numberwithin{equation}{section}
\numberwithin{figure}{section}

\begin{document}

\title{Rigidity for the perimeter inequality under Schwarz symmetrisation}

\author{Georgios Domazakis}
\date{\today}
\address{Department of Mathematics, University of Sussex, Pevensey 2, BN1 9QH, Brighton, UK}%
\email{g.domazakis@sussex.ac.uk}
\maketitle
\begin{abstract}
In this paper, we give necessary and sufficient conditions for the rigidity of perimeter inequality under Schwarz symmetrisation. The term \textit{rigidity} refers to the situation in which the equality cases are only obtained by translations of the symmetric set. In particular, we prove that  the sufficient conditions for rigidity provided in \cite{fusco2013stability}, are also  necessary. 

\end{abstract}

\section{Introduction}

Symmetrisation procedures have a wide range of applications in modern analysis, geometric variational problems and optimisation. Understanding the behaviour of functional and perimeter inequalities under symmetrisation  
 allows to prove the existence of symmetric minimisers of geometric variational problems, and to provide comparison principles for solutions of PDEs (see, for instance \cite{kawohl_rearrangements_1985,kawohl1986isoperimetric,Talenti1986,trombetti1997convex} and the references therein).

Examples of set symmetrisations under which the volume is preserved and the perimeter  does not increase include Steiner symmetrisation, Ehrhard symmetrisation, circular and spherical symmetrisation. We say that rigidity holds for a perimeter inequality if the set of extremals is trivial. Showing rigidity can lead to proving the uniqueness of minimisers of variational problems.
For example, proving the rigidity of  Steiner's inequality for convex sets was substantial in the celebrated proof  of the Euclidean isoperimetric inequality by De Giorgi  (see, \cite{de1958sulla,de2007selected}).

Later on, the study of rigidity was revived in the seminal paper of Chleb\'ik, Cianchi and Fusco (see \cite{chlebik2005perimeter}), where the authors gave the sufficient conditions for rigidity of Steiner's inequality, also for sets that are not convex. 
Henceforth, necessary and sufficient conditions for rigidity for Steiner's inequality have been obtained in 
\cite{cagnetti2014rigidity} in the case where the distribution function is a Special function of Bounded Variation with locally finite jump set. In the Gauss space,  necessary and sufficient conditions for rigidity of Ehrhard's  inequality are given in \cite{cagnetti2017essential}. In the last two papers, the results are stated in terms of \textit{essential connectedness}.  For an expository article of the aforementioned rigidity results, we refer to 
\cite{cagnetti2019rigidity}. In \cite{cagnetti2020rigidity}, the authors provided the necessary and sufficient conditions for rigidity for perimeter inequality under spherical symmetrisation, while in \cite{perugini2021rigidity}, sufficient conditions for rigidity have been given for the anisotropic Steiner's perimeter inequality.  We further point out that, regarding the smooth case, the authors in \cite{morgan2011steiner} proved sufficient conditions for rigidity  of perimeter inequality  in \textit{warped products}, for  a wide class of symmetrisations, including Steiner, Schwarz and spherical symmetrisation. 

\nocite{cianchi2002functions}
\nocite{serrin1971symmetry,polya1951isoperimetric,morgan2013existence,kawohl1986isoperimetric,gidas1979symmetry}

 The literature about Steiner's perimeter inequality of a higher codimension is less explored. Particularly, sufficient conditions for rigidity for any codimension have been provided in \cite{fusco2013stability}, through a comprehensive analysis of the barycenter function. 
  The problem of a complete characterisation (that is, necessary \textit{and} sufficient conditions) for the rigidity of generic higher codimensions, however, remains  open.

A special case of interest is where the codimension is equal to $(n-1)$. In this case, Steiner's symmetrisation of codimension $(n-1)$ is usually referred to as \textit{Schwarz symmetrisation}.

 The purpose of this paper is to provide necessary and sufficient conditions for rigidity of equality cases for the perimeter inequality under Schwarz symmetrisation.
 In particular, we prove that the sufficient conditions for rigidity shown in  \cite{fusco2013stability} are also  necessary. Our results are established by following techniques developed in \cite{cagnetti2020rigidity}.

 In the remainder of this introductory section, we recall the setting of the problem, and we state our main results.

\subsection{Schwarz symmetrisation}
For $n \geq 2$  with $n \in \mathbb{N}$, we label each point $x \in \mathbb{R}^{n}$  as   $x= (z,w)$, where $z \in \mathbb{R}$ and $w \in \mathbb{R}^{n-1}$.  

Given a measurable  set $E \subset \mathbb{R}^{n}$ and $z \in \mathbb{R}$, we define the \textit{$(n-1)$-dimensional  slice} of $E$ at $z$ as 
\begin{equation} \label{slice}
E_{z} := \{ w \in \mathbb{R}^{n-1}: (z,w) \in E \}.
\end{equation}

For  a Lebesgue measurable function $\ell : \mathbb{R} \rightarrow [0, \infty),$ we say that the set $E$ is $\ell$-\textit{distributed} if 
\begin{equation}\label{elldistributed}
\ell(z) = \mathcal{H}^{n-1} (E_{z}) \quad \mbox{ for } \mathcal{H}^{1} \mbox{-a.e. } z \in \mathbb{R},
\end{equation}
where $\mathcal{H}^{n-1}$ denotes the $(n-1)$-dimensional Hausdorff measure in $\mathbb{R}^{n}$.

We can associate to $\ell$ the function $r_{\ell} : \mathbb{R} \rightarrow [0,\infty)$, which is such that 

\[
\ell(z) = \mathcal{H}^{n-1} \left( B^{n-1} (0 , r_{\ell} (z) \right),  \quad \mbox{ for } \mathcal{H}^{1} \mbox{-a.e. } z \in \mathbb{R},
\]
where $B^{n-1} (w, \rho)$ denotes the open ball in $\mathbb{R}^{n-1}$ with radius $\rho$ and centered at $w \in \mathbb{R}^{n-1}$.

Note that $r_{\ell}(z)$ is the radius of an $(n-1)$-dimensional ball whose measure is $\ell(z)$, and can  be explicitly written as

\begin{equation} \label{aktina}
r_{\ell}(z)=\left(\frac{\ell(z)}{\omega_{n-1}}\right)^{\frac{1}{n-1}} \quad \mbox{ for } \mathcal{H}^{1} \mbox{-a.e. } z \in \mathbb{R},
\end{equation}
where we set $\omega_{n-1}:= \mathcal{H}^{n-1}(B^{n-1}(0,1))$.

If $E \subset \mathbb{R}^{n}$ is $\ell$-distributed, then the \textit{Schwarz symmetric} set $F_{\ell}$ of $E$ with respect to the axis $\{ w= 0 \}$ is defined as 
\begin{equation}\label{definition_of_Fell}
F_{\ell}:= \left\{ x= (z,w) \in \mathbb{R} \times \mathbb{R}^{n-1}: |w| < r_{\ell} (z) \right\};
\end{equation}
see Figure \ref{symmetralfig_label}.

This is the $\ell$-distributed set whose cross sections are $(n-1)$-dimensional open balls centred at the $z$ axis. We notice that the Schwarz symmetric set $F_{\ell}$ of an $\ell$-distributed set $E$ depends only on the function $\ell$, and not  on the particular $\ell$-distributed set $E$ under consideration.
\tdplotsetmaincoords{85}{113}
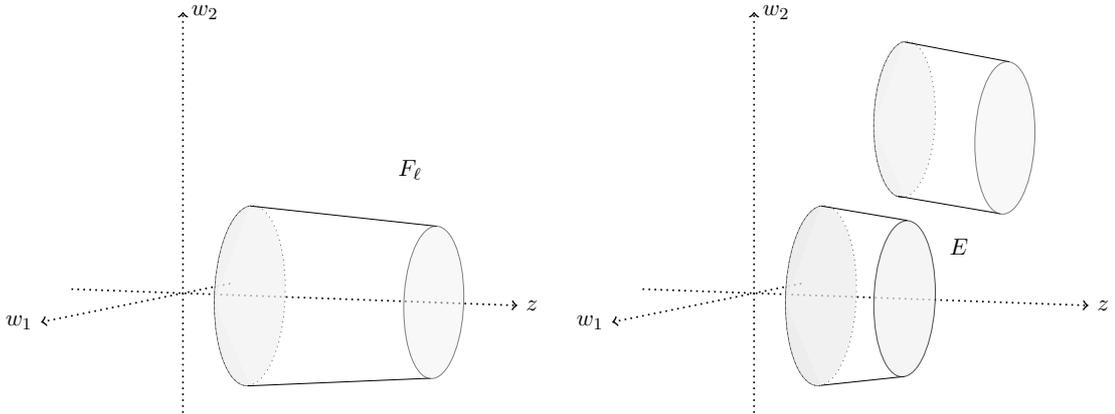
\begin{figure}[t]
\centering
\resizebox{1\textwidth}{!}{

\subfigure
    {%
	\begin{tikzpicture}
		[tdplot_main_coords,
		cube/.style={very thin,black},
		grid/.style={very thin,gray },
		axis/.style={->,black,thick, dotted}]
		
		\draw[axis] (0,-2,0) -- (0,6,0) node[anchor=west]{$z$};
		\draw[axis] (-2,0,0) -- (6,0,0)   node[anchor=east]{$w_{1}$};
		\draw[axis] (0,0,-2) -- (0,0,4.7) node[anchor=west]{$w_{2}$};
		
		\begin{scope}[canvas is xz plane at y=1.2]

			\draw[fill=gray!10, opacity=0.7](1.5,0) arc[x radius=1.5, y radius=1.5, start angle=0, end angle=90];
			
			\draw[fill=gray!10, opacity=0.7] (1.5,0) arc[x radius=1.5, y radius=1.5, start angle=0, end angle=-90];
			
			\draw [dotted,fill=gray!10, opacity=0.7] (1.5,0) arc[x radius=1.5, y radius=1.5, start angle=0, end angle=180];
			
			\draw [dotted,fill=gray!10, opacity=0.7] (1.5,0) arc[x radius=1.5, y radius=1.5, start angle=0, end angle=-180];
			
		\end{scope}
        
		\node at (1,4.5,2.3) {$F_{\ell}$}; 
		
	    \draw[-](0,1.2,1.5)-- 	(0,4.555,1.275);
        \draw[-](0,1.16,-1.51)-- 	(0,4.44,-1.27);
		
		\begin{scope}[canvas is xz plane at y=4.5]
			
			\draw[fill=gray!10, opacity=0.5]  (0,0) circle (1.27);
			
		\end{scope}	
	\end{tikzpicture}
	}
\subfigure
    {
	\begin{tikzpicture}
		[tdplot_main_coords,
		cube/.style={very thin,black},
		grid/.style={very thin,gray },
		axis/.style={->,black,thick, dotted}]
		


		\draw[axis] (0,-2,0) -- (0,6,0) node[anchor=west]{$z$};
		\draw[axis] (-2,0,0) -- (6,0,0)   node[anchor=east]{$w_{1}$};
		\draw[axis] (0,0,-2) -- (0,0,4.7) node[anchor=west]{$w_{2}$};
		

		\begin{scope}[canvas is xz plane at y=1.2]
			
			\draw[fill=gray!15, opacity=0.7](1.5,0) arc[x radius=1.5, y radius=1.5, start angle=0, end angle=90];
			\draw[fill=gray!15, opacity=0.7] (1.5,0) arc[x radius=1.5, y radius=1.5, start angle=0, end angle=-90];
			\draw [dotted,fill=gray!15, opacity=0.7] (1.5,0) arc[x radius=1.5, y radius=1.5, start angle=0, end angle=180];
			\draw [dotted,fill=gray!15, opacity=0.7] (1.5,0) arc[x radius=1.5, y radius=1.5, start angle=0, end angle=-180];
			
		\end{scope}
		
		
		\node at (1,4.1,1) {$E$}; 
	    
		
		\draw[-](0,1.2,1.5)-- (0,2.75,1.31);
		\draw[-](0,1.2,-1.5)-- (0,2.7,-1.3);

		\begin{scope}[canvas is xz plane at y=2.7]
			
			\draw[name path=circ1, fill=gray!10, opacity=0.7] (0,0) circle (1.3);

		\end{scope}
		
		
		\begin{scope}[canvas is xz plane at y=2.7]

			\draw[fill=gray!10, opacity=0.7 ](1.3,3) arc[x radius=1.3, y radius=1.3, start angle=0, end angle=90];
			\draw[fill=gray!10, opacity=0.7](1.3,3) arc[x radius=1.3, y radius=1.3, start angle=0, end angle=-90];
			\draw [dotted,	fill=gray!10, opacity=0.7] (1.3,3) arc[x radius=1.3, y radius=1.3, start angle=0, end angle=180];
			\draw [dotted,fill=gray!10, opacity=0.7](1.3,3) arc[x radius=1.3, y radius=1.3, start angle=0, end angle=-180];
		\end{scope}

		
		\draw[] (0,2.7,3+1.292)-- (0,4.58,1.325+2.7);
		\draw[] (0,2.58,3-1.29)-- (0,4.42,-1.22+2.7);
		
		\begin{scope}[canvas is xz plane at y=4.5]
			
			\draw[fill=gray!10, opacity=0.5]  (0,2.75) circle (1.27);
			
		\end{scope}	
	\end{tikzpicture}
    }
}

\caption{The symmetric set $F_{\ell}$ (left) of an $\ell$-distrubuted set $E$ (right) in case of $n=3$. Note that, in general the slices of the set $E$ do not need to be disks.
} 
\label{symmetralfig_label}
\end{figure}

 Due to Fubini's theorem, Schwarz symmetrisation preserves the volume, i.e. if $E$ is $\ell$-distributed and $\mathcal{H}^{n} (E) < \infty$, it turns out that $\mathcal{H}^{n} (E) =\mathcal{H}^{n}(F_{\ell}).$ 
 Moreover, the perimeter inequality under Schwarz symmetrisation holds,  that is 
\begin{equation} \label{perimeter inequality}
P( F_{\ell} ) \leq P( E ) \mbox{ for every $\ell$-distributed set } E \subset \mathbb{R}^{n}.
\end{equation}
Here, $P(E)$ stands for the perimeter of $E$ in $\mathbb{R}^{n}$ (see Section \ref{setsofFP}).

The inequality \eqref{perimeter inequality} is well-known in the literature (see, for instance, \cite{brock2000approach}, where this is proved through a careful approximation by polarisations). In \cite{fusco2013stability}, one can find an alternative and direct proof, which allowed the authors to give sufficient conditions for rigidity  of  Steiner's inequality of a general higher codimension $k$, where $1< k \leq n-1$.

\subsection{Rigidity for perimeter inequality under Schwarz symmetrisation}
We shall now describe the main objective of the present paper. Given a Lebesgue measurable function $\ell: \mathbb{R} \rightarrow [0,\infty)$, such that $F_{\ell}$ is a set of finite perimeter and finite volume, we define the class of equality cases of \eqref{perimeter inequality} as

\begin{equation} \label{extremals class}
\mathcal{K}(\ell) = \{ E \subset \mathbb{R}^{n} : E \mbox{ is $\ell$-distributed and } P (F_{\ell}) = P(E) \}.
\end{equation}  \\ 
Due to  the invariance of the perimeter under translations along a direction $\tau \in \mathbb{R}^{n-1}$, as well as the definition of the  symmetric set  $F_{\ell}$, the following inclusion is always true:

\begin{equation}\label{class inclusion}
\mathcal{K}(\ell) \supset \{E \subset \mathbb{R}^{n} : \mathcal{H}^{n} ( E \triangle ( F_\ell + (0,\tau))) = 0 \mbox{ for some } \tau \in \mathbb{R}^{n-1}\},
\end{equation} \\ 
where $\triangle$ denotes the symmetric difference of sets.
We say that \textit{rigidity} holds  for \eqref{perimeter inequality} if  the opposite inclusion is also satisfied, i.e.

\begin{equation}\label{class rigidity}
\mathcal{K}(\ell) = \{E \subset \mathbb{R}^{n} :  \mathcal{H}^{n} ( E \triangle ( F_\ell + (0,\tau)) = 0 \mbox{ for some } \tau\in \mathbb{R}^{n-1}\}. \tag{$\mathcal{R}\mathcal{S}$}
\end{equation} 

\subsection{State of the art}
Let us now give an account of the available results in the literature for the rigidity of \eqref{perimeter inequality}. 
In general, not all equality cases of \eqref{perimeter inequality} can be written as a translation of the symmetric set $F_{\ell}$. This can happen, for instance, if the (reduced) boundary  $\partial^{*} F_{\ell}$ of $F_{\ell}$ contains flat vertical parts. In such a case, we can find an $\ell$-distributed set $E$ which preserves perimeter under symmetrisation, and it is not equivalent to (a translation of) the symmetric set 
$F_{\ell}$; see Figure \ref{ell jump counter fig_label}. 
\tdplotsetmaincoords{85}{113}

\begin{figure}[h]
\centering
\resizebox{1.12\textwidth}{!}
{
\subfigure
    {
     \begin{tikzpicture}
	    [tdplot_main_coords,
	    cube/.style={very thin,black},
	    grid/.style={very thin,gray },
	    axis/.style={->,black,thick, dotted}]
        

        \draw[axis] (0,-2,0) -- (0,6,0) node[anchor=west]{$z$}; 
        \draw[axis] (-2,0,0) -- (6,0,0)   node[anchor=east]{$w_{1}$};
        \draw[axis] (0,0,-2) -- (0,0,4) node[anchor=west]{$w_{2}$};

        \begin{scope}[canvas is xz plane at y=1.2]
        
        	\draw[fill=gray!15, opacity=0.7](1.5,0) arc[x radius=1.5, y radius=1.5, start angle=0, end angle=90];
	
	        \draw[fill=gray!15, opacity=0.7] (1.5,0) arc[x radius=1.5, y radius=1.5, start angle=0, end angle=-90];
	        
	        \draw [dotted,fill=gray!15, opacity=0.7] (1.5,0) arc[x radius=1.5, y radius=1.5, start angle=0, end angle=180];
	
	        \draw [dotted,fill=gray!15, opacity=0.7] (1.5,0) arc[x radius=1.5, y radius=1.5, start angle=0, end angle=-180];
        \end{scope}
        
        \draw[-](0,1.2,1.5)-- (0,2.7,1.3);
        \draw[-](0,1.2,-1.5)-- (0,2.7,-1.3);

        \begin{scope}[canvas is xz plane at y=2.7]
	        
	        \draw[name path=circ1, fill=gray!10, opacity=0.7] (0,0) circle (1.3);

	     ;
        \end{scope}

        \begin{scope}[canvas is xz plane at y=2.7]

	        \draw[fill=white, opacity=0.7 ](0.8,0) arc[x radius=0.8, y radius=0.45, start angle=0, end angle=90];
	        
	        \draw[fill=white, opacity=0.7](0.8,0) arc[x radius=0.8, y radius=0.45, start angle=0, end angle=-90];
	        
	        \draw [dotted,fill=white, opacity=0.7] (0.8,0) arc[x radius=0.8, y radius=0.45, start angle=0, end angle=180];
	    
	        \draw [dotted,fill=white, opacity=0.7](0.8,0) arc[x radius=0.8, y radius=0.45, start angle=0, end angle=-180];

        	\draw  node[fill,circle,scale=0.2]{} (0,1);
        	
		    \node at (0,-0.2) {$\tilde{z}$} ;

        \end{scope}
        
        
        \node at (2,3.5,2.5) {$F_{\ell}$};

        
        \draw[] (0,2.7,0.45)-- (0,4,0.8);
        
        \draw[] (0,2.565,-0.45)-- (0,3.93,-0.8);

        \begin{scope}[canvas is xz plane at y=4]
	        
	        \draw[fill=gray!10, opacity=0.5]  (0,0) circle (0.8);
	
        \end{scope}	    
\end{tikzpicture}

    }

\subfigure
    {
    \begin{tikzpicture}
	[tdplot_main_coords,
	cube/.style={very thin,black},
	grid/.style={very thin,gray },
	axis/.style={->,black,thick, dotted}]
	
	
	\draw[axis] (0,-2,0) -- (0,6,0) node[anchor=west]{$z$};
	\draw[axis] (-2,0,0) -- (6,0,0)   node[anchor=east]{$w_{1}$};
	\draw[axis] (0,0,-2) -- (0,0,4) node[anchor=west]{$w_{2}$};

	    \begin{scope}[canvas is xz plane at y=1.2]

		    \draw[fill=gray!15, opacity=0.7](1.5,0) arc[x radius=1.5, y radius=1.5, start angle=0, end angle=90];
		  
		    \draw[fill=gray!15, opacity=0.7] (1.5,0) arc[x radius=1.5, y radius=1.5, start angle=0, end angle=-90];
		    
		    \draw [dotted,fill=gray!15, opacity=0.7] (1.5,0) arc[x radius=1.5, y radius=1.5, start angle=0, end angle=180];
		
		    \draw [dotted,fill=gray!15, opacity=0.7] (1.5,0) arc[x radius=1.5, y radius=1.5, start angle=0, end angle=-180];
		    
	    \end{scope}
	   
	    \draw[-](0,1.2,1.5)-- (0,2.7,1.3);
	   
	    \draw[-](0,1.2,-1.5)-- (0,2.7,-1.3);

    	\begin{scope}[canvas is xz plane at y=2.7]
	
		    \draw[name path=circ1, fill=gray!10, opacity=0.7] (0,0) circle (1.3);
		
	    	\draw  node[fill,circle,scale=0.2]{} (0,1);
		   
		    \node at (.3,2.2,5.2) {$E$} ;
		
    	\end{scope}

	       \begin{scope}[canvas is xz plane at y=2.7]
	
	    	\draw[fill=gray!10, opacity=0.7 ](0.8,0.75,0.45) arc[x radius=0.8, y radius=0.45, start angle=0, end angle=90];
	    	
		    \draw[fill=gray!10, opacity=0.7](0.8,0.75,0.45) arc[x radius=0.8, y radius=0.45, start angle=0, end angle=-90];
		    
		    \draw [dotted,	fill=gray!10, opacity=0.7] (0.8,0.75,0.45) arc[x radius=0.8, y radius=0.45, start angle=0, end angle=180];
		    
		    \draw [dotted,fill=gray!10, opacity=0.7](0.8,0.75,0.45) arc[x radius=0.8, y radius=0.45, start angle=0, end angle=-180];
		
	    \end{scope}

 
       \draw[] (0,2.711,1.2)-- (0,4.01,1.552);
        
        \draw[] (0,2.58,0.291)-- (0,3.95,-0.05);


 \begin{scope}[canvas is xz plane at y=4]
	        
	        \draw[fill=gray!10, opacity=0.5]  (0,0.75,0) circle (0.8);
	
        \end{scope}	
    \begin{scope}[canvas is xz plane at y=2.7]

	    	\draw  node[fill,circle,scale=0.2]{} (0,1);
		    \node at (0,-0.2) {$\tilde{z}$} ;

    	\end{scope}     
    \end{tikzpicture}
    }

}

\caption{Rigidity \eqref{class rigidity} fails, since the (reduced) boundary $\partial^{*} F_{\ell}$ of $F_{\ell}$ has a non-negligible flat vertical part, thus violating \eqref{no flat parts symmetric}. Note that the function $\ell$ is discontinuous at $\tilde{z}$, so that also \eqref{sobolev_assumption1} is violated.}
\label{ell jump counter fig_label}

\end{figure}
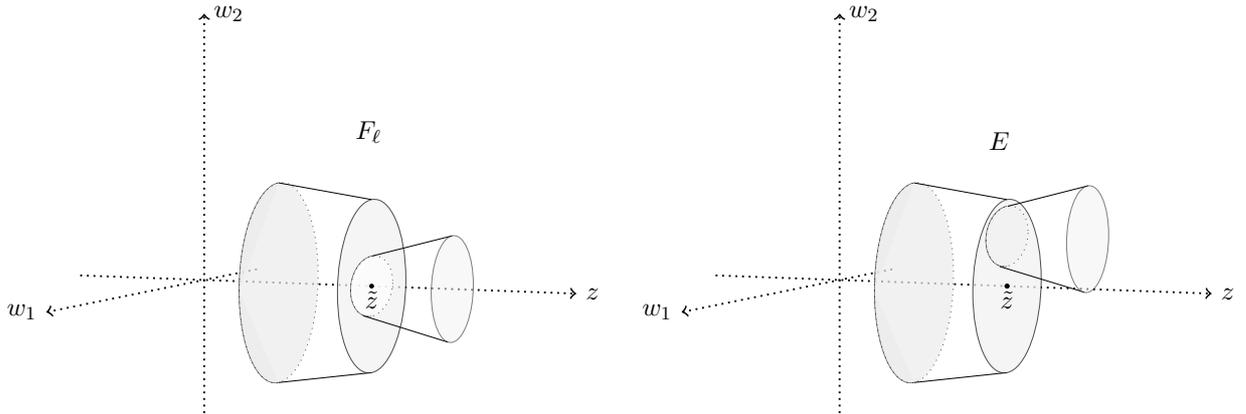

In order to rule out this issue,  the authors in \cite{fusco2013stability} localised the problem, by considering an open set $\Omega \subset \mathbb{R}$, and imposing the following condition:

\begin{equation} \label{no flat parts symmetric}
\mathcal{H}^{n-1} ( \{ (z,w) \in \partial^{*} F_{\ell} : \nu_{w}^{F_{\ell}} (z,w) = 0 \} \cap (\Omega \times \mathbb{R}^{n-1} ) ) =0,
\end{equation}  \\
where $\nu_{w}^{F_{\ell}}(z,w)$ denotes the $w$-component of the measure-theoretic outer unit normal to the symmetric set $F_{\ell}.$ It turns out that \eqref{no flat parts symmetric} is related to the regularity of the function $\ell$. Note that, in general, if $E$ is a set of finite perimeter in $\mathbb{R}^{n}$, then either $F_{\ell}$ is equivalent to $\mathbb{R}^{n}$, or $\ell$ is a function of Bounded Variation in $\mathbb{R}$ (see Proposition \ref{ell_regular}).

In \cite[Proposition~3.5]{fusco2013stability}, the authors showed that \eqref{no flat parts symmetric} is equivalent to asking that $\ell$ is a Sobolev function in $\Omega$, as explained below.

\begin{proposition} \label{fuscoreg}
Let $\ell: \mathbb{R}\rightarrow [0, \infty)$ be a measurable function, such that $F_{\ell}$ is a set of finite perimeter and finite volume in $\mathbb{R}^{n}$ and  let $\Omega \subset \mathbb{R}$ be an open set. Then 
\[
\mathcal{H}^{n-1} ( \{ (z,w) \in \partial^{*} F_{\ell} : \nu_{w}^{F_{\ell}} (z,w) = 0 \} \cap (\Omega \times \mathbb{R}^{n-1} ) ) =0
\]
if and only if 
\begin{equation} \label{sobolev_assumption1}\ell \in W^{1,1} (\Omega).
\end{equation}
\end{proposition}
\vspace{0.2cm}

Even if condition \eqref{no flat parts symmetric} (or, equivalently, \eqref{sobolev_assumption1}) is satisfied, rigidity can still be violated.  In particular, this can  happen when  the symmetric set $F_{\ell}$ \textit{is not connected in a suitable measure-theoretic way}, despite the fact that it can be connected from a topological point of view; see Figure \ref{ell zero counter fig_label}.
\tdplotsetmaincoords{85}{113}
\begin{figure}[ht]
\centering
\resizebox{1\textwidth}{!}
{

\
\subfigure
	{%

    \begin{tikzpicture} 
	    [tdplot_main_coords,
	    cube/.style={very thin,black},
	    grid/.style={very thin,gray },
	    axis/.style={->,black,thick, dotted}]
	
        \draw[axis] (0,-2,0) -- (0,7,0) node[anchor=west]{$z$};
        \draw[axis] (-2,0,0) -- (6,0,0)   node[anchor=east]{$w_{1}$};
        \draw[axis] (0,0,-2) -- (0,0,4) node[anchor=west]{$w_{2}$};

        \begin{scope}[canvas is xz plane at y=0.9]
	
    	    \draw[fill=gray!10, opacity=0.6](1.2,0) arc[x radius=1.5, y radius=1.5, start angle=0, end angle=90];
    	    
    	    \draw[fill=gray!10, opacity=0.6] (1.2,0) arc[x radius=1.5, y radius=1.5, start angle=0, end angle=-90];
    	    
	        \draw [dotted,fill=gray!10, opacity=0.6] (1.2,0) arc[x radius=1.5, y radius=1.5, start angle=0, end angle=180];
	        
	        \draw [dotted,fill=gray!10, opacity=0.6] (1.2,0) arc[x radius=1.5, y radius=1.5, start angle=0, end angle=-180];

        \end{scope}


        \draw[-](0,1.04,1.555)-- (0,3.7,0);
        \draw[-](0,1,-1.5)-- (0,3.68,0);

        \begin{scope}[canvas is xz plane at y=3.7]

	        \draw  node[fill,color=black,circle,scale=0.2]{} (0,1);
	        
        	\node at (0,-0.3) {$\tilde{z}$} ;

        \end{scope}


        \draw[] (0,3.71,0)-- (0,5.16,1.19);
        
        \draw[] (0,3.68,0)-- (0,5.06,-1.19);
        
        \node at (2,5, 2.2) {$F_\ell$} ;

        \begin{scope}[canvas is xz plane at y=5.2]
	
	        \draw[fill=gray!10, opacity=0.6]  (0,0) circle (1.2);
	
        \end{scope}	
	
	\end{tikzpicture} 
    }

\subfigure
    {%
	\begin{tikzpicture} 
		[tdplot_main_coords,
		cube/.style={very thin,black},
		grid/.style={very thin,gray },
		axis/.style={->,black,thick, dotted}]
		
		\draw[axis] (0,-2,0) -- (0,7,0) node[anchor=west]{$z$};
		\draw[axis] (-2,0,0) -- (6,0,0)   node[anchor=east]{$w_{1}$};
		\draw[axis] (0,0,-2) -- (0,0,4) node[anchor=west]{$w_{2}$};

\begin{scope}[canvas is xz plane at y=0.9]
	
    	    \draw[fill=gray!10, opacity=0.6](1.2,0) arc[x radius=1.5, y radius=1.5, start angle=0, end angle=90];
    	    
    	    \draw[fill=gray!10, opacity=0.6] (1.2,0) arc[x radius=1.5, y radius=1.5, start angle=0, end angle=-90];
    	    
	        \draw [dotted,fill=gray!10, opacity=0.6] (1.2,0) arc[x radius=1.5, y radius=1.5, start angle=0, end angle=180];
	        
	        \draw [dotted,fill=gray!10, opacity=0.6] (1.2,0) arc[x radius=1.5, y radius=1.5, start angle=0, end angle=-180];

        \end{scope}


        \draw[-](0,1.04,1.555)-- (0,3.7,0);
        \draw[-](0,1,-1.5)-- (0,3.68,0);

        \begin{scope}[canvas is xz plane at y=3.7]
	        \draw  node[fill,color=black,circle,scale=0.2]{} (0,1);
        	\node at (0,-0.3) {$\tilde{z}$} ;
	
        \end{scope}


	    \draw[] (0,3.7,2.2)-- (0,5.16,1.19+2.2);
	    
	    \draw[] (0,3.68,2.2)-- (0,5.06,-1.19+2.2);
	    
	    \node at (2,3.3,3) {$E$} ;

	    \begin{scope}[canvas is xz plane at y=5.2]
	
		    \draw[fill=gray!10, opacity=0.6]  (0,2.2) circle (1.2);
		
	    \end{scope}	
	\end{tikzpicture} 
    }
}
\caption{Rigidity \eqref{class rigidity} fails, since the set $\{ \ell^{\wedge} >0 \}$ is disconnected by a point $\tilde{z} \in \mathbb{R}$, where $\ell(\tilde{z})=0$, thus, violating \eqref{barchiesi_ell positive}.}
\label{ell zero counter fig_label}
\end{figure}
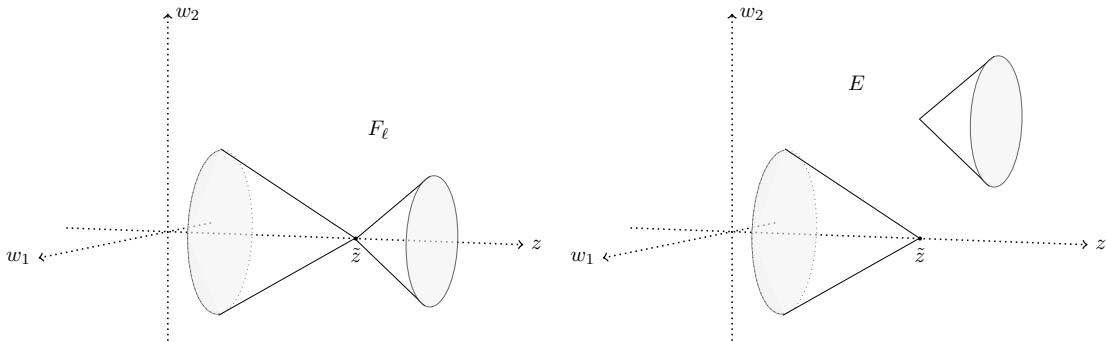

Note that, once condition  \eqref{no flat parts symmetric} (or, equivalently, \eqref{sobolev_assumption1}) is imposed, we have that $\ell \in W^{1,1} (\Omega)$, and since $\Omega$ is a one-dimensional set,  $\ell$ is absolutely continuous in $\Omega$. Therefore the condition imposed in \cite{fusco2013stability} to rule out situations as in Figure \ref{ell zero counter fig_label} can be written as 
\begin{equation} \label{barchiesi_ell positive}
\ell(z) >0 \mbox{ for all } z \in \Omega,
\end{equation}
see \cite[Condition~(1.4)]{fusco2013stability}.

 It turns out that \eqref{no flat parts symmetric} and \eqref{barchiesi_ell positive} are sufficient for rigidity (see \cite[Theorem~1.2]{fusco2013stability}), as explained below.

 \begin{theorem}\label{nico-marco-filippo}
Let $\ell : \mathbb{R} \rightarrow [0, \infty)$ be a measurable function, such that $F_{\ell}$ is a set of finite perimeter and finite volume. Let $\Omega \subset \mathbb{R}$ be a connected open set, and suppose that \eqref{no flat parts symmetric} and \eqref{barchiesi_ell positive} are satisfied. If
\[
P( F_{\ell} ; \Omega \times \mathbb{R}) = P (E ; \Omega \times \mathbb{R} ),
\]
then $E \cap ( \Omega \times \mathbb{R} ) $ is equivalent to (a translation along $\mathbb{R}^{n-1}$) of $F_{\ell} \cap ( \Omega \times \mathbb{R})$. Here, $P(E; \Omega \times \mathbb{R})$ denotes the relative perimeter of $E$ in $\Omega \times \mathbb{R}$.
 \end{theorem}
\subsection{The main result}

Our contribution is to show that conditions \eqref{no flat parts symmetric} and \eqref{barchiesi_ell positive} are also \textit{necessary} for rigidity. As we have already observed, the proof of the Theorem \ref{nico-marco-filippo} requires the localisation of the problem in an open and connected set $\Omega \subset \mathbb{R}$, to impose the condition \eqref{sobolev_assumption1}. We will show that this can be avoided.
We also notice that, if $F_{\ell}$ is a set of finite perimeter and finite volume, in general, we only have that $\ell \in BV(\mathbb{R})$ and this means that $\ell$ may be discontinuous. Therefore, we need to rephrase condition \eqref{barchiesi_ell positive} in terms of the approximate $\liminf$ $\ell^{\wedge}$ of $\ell$ at every point $z \in \mathbb{R}$, see Section \ref{preliminaries}. We are now able to state our main result. Below, $\mathring{J}$ denotes the interior of $J$.

\begin{theorem} \label{complete characterisation}
Let $\ell: \mathbb{R}\rightarrow [0, \infty)$ be a measurable function,  such that $F_{\ell}$ is a set of finite perimeter and finite volume. Then, the following statements are equivalent:
\begin{enumerate}
    \item[(i)] \eqref{class rigidity} holds true; \label{rigidity_stament}
    \vspace{.2cm}
    \item[(ii)] $\{ \ell^{\wedge}  >0\}$ is a (possibly unbounded) interval $J$ and $\ell \in W^{1,1}(\mathring{J})$. \label{assumptions_statement}
\end{enumerate}
\end{theorem}

As we have already pointed out, the proof of the direction $(ii) \Longrightarrow (i)$ of Theorem \ref{complete characterisation} relies on the proof of \cite[Theorem 1.2]{fusco2013stability}. We will prove that the converse  $(i) \Longrightarrow (ii)$ is also true. 
 We would like to emphasise that our approach does not lie on the comprehensive use of a general perimeter formula for sets $E \subset \mathbb{R}^{n}$ satisfying equality in \eqref{perimeter inequality}, as it appears in \cite{cagnetti2017essential}. On the contrary, inspired by the techniques developed in \cite{cagnetti2020rigidity}, we analyse the properties of the function $\ell$ and we provide a careful study  of the transformations that can be applied  on the symmetric set $F_{\ell}$, without creating any perimeter contribution.

To this end, the rest of the paper is structured as follows. In Section \ref{preliminaries}, we fix the notation, build the necessary background  and we gather some preliminary results that appeared in the literature.  In Section \ref{results}, we show the direction $(i) \Longrightarrow (ii)$ of Theorem \ref{complete characterisation}, by  studying the properties of the distribution function $\ell,$ and exploiting counterexamples where rigidity is violated. 

\section{Background and proof of the Theorem 1.3 $(ii) \Longrightarrow (i)$ }\label{preliminaries}

In this section, we will recall the necessary machinery, which will be used throughout the paper. The interested reader could refer to
\cite{ambrosio2000functions,fusco2013stability,evans2018measure,federer2014geometric,maggi2012sets,simon1983lectures}.

We fix $n\in \mathbb{N}$, with $n \geq 2$. For each $x \in \mathbb{R}^{n}$, we write $x=(z,w)$, with $z \in \mathbb{R}$ and $w \in \mathbb{R}^{n-1}$. The standard Euclidean norm will be denoted by $|\cdot|$ in $\mathbb{R}, \ \mathbb{R}^{n-1}$ or  $\mathbb{R}^{n}$ depending on the context. For $1 \leq m \leq n$, we will denote the $m$-dimensional Hausdorff measure in $\mathbb{R}^{n}$ by $\mathcal{H}^{m}.$ For every radius $\rho>0$ and $x \in \mathbb{R}^{n}$ we write $B_{\rho}(x)$ for the open ball of $\mathbb{R}^{n}$ with radius $\rho$ and centered at $x.$ The volume of the unit ball in $\mathbb{R}^{n}$ is denoted as $\omega_{n}$, i.e. $\omega_{n}:= \mathcal{H}^{n} (B_{1}(0))$. 
Note that throughout the paper, in case of balls in different dimensions, we will denote the corresponding ball in  dimension $m$ with radius $\rho$ centred at $w \in \mathbb{R}^{m}$  by writing $B^{m}(w,\rho)$.

Now, for $x \in \mathbb{R}^{n}$ and $\nu \in \partial B_{1}(0)$, we set
\[
H_{x,\nu}^{+}= \left\{ y\in \mathbb{R}^{n} : \langle (y-x), \nu \rangle\geq 0 \right\}
\]
and
\[
H_{x,\nu}^{-}  = \left\{ y\in \mathbb{R}^{n} : \langle (y-x), \nu \rangle \leq 0 \right\}.
\]

Let  $\{E_{j} \} _{j \in \mathbb{N}}$ be a sequence of Lebesgue measurable sets in $\mathbb{R}^{n}$ with $\mathcal{H}^{n} (E_{j}) < \infty$ for every $j \in \mathbb{N}$, and let $E \subset \mathbb{R}^{n}$ be a Lebesgue measurable set with $\mathcal{H}^{n} (E) < \infty$. We say that $\{ E_{j} \}_{j \in \mathbb{N}}$ \textit{converges to} $E$ as $j \rightarrow \infty$ and we write 
\[
E_{j} \rightarrow E  \quad \mbox{ if } \quad  \mathcal{H}^{n}( E_{j} \triangle E ) \rightarrow 0 \mbox{ as } j \rightarrow \infty, 
\]
where $\triangle$ stands for the symmetric difference of sets.  Additionally, if $E_{1}, E_{2} \subset \mathbb{R}^{n}$ are Lebesgue measurable sets, we say that 
\[
E_{1} \subset_{\mathcal{H}^{n}} E_{2} \quad \mbox{if} \quad  \mathcal{H}^{n} ( E_{1} \backslash E_{2} ) =0,
\]
and
\[
E_{1} =_{\mathcal{H}^{n}} E_{2} \quad \mbox{if} \quad \mathcal{H}^{n} ( E_{1} \triangle E_{2} ) =0.
\]
Moreover, the characteristic function of a Lebesgue measurable set $E \subset \mathbb{R}^{n}$ will be denoted by $\chi_{E}$. 

\subsection{Density points}
Let $E \subset \mathbb{R}^{n}$ be a Lebesgue measurable set and $ x \in \mathbb{R}^{n}.$ We define the \textit{lower} and \textit{upper} $n$-\textit{dimensional densities of $E$ at $x$} as 
\[
\theta_{*} (E,x) = \liminf_{\rho \rightarrow 0^{+}} \frac{ \mathcal{H}^{n} ( E \cap B_{\rho}(x))}{ \omega_{n} \rho^{n} }, \quad  \mbox{ and } \quad \theta^{*} (E,x) = \limsup_{\rho \rightarrow 0^{+}} \frac{ \mathcal{H}^{n} ( E \cap B_{\rho}(x))}{ \omega_{n} \rho^{n} },
\]
respectively.
The maps $x \longmapsto \theta_{*} (E,x)$ and  $x \longmapsto \theta^{*} (E,x)$ are Borel functions (even in case where $E$ is Lebesgue non-measurable) and they coincide $\mathcal{H}^{n}$-a.e. in $\mathbb{R}^{n}.$ Hence, the $n$-dimensional density of $E$ at $x$ is defined as the Borel function
\[
\theta(E,x) =\lim_{\rho \rightarrow 0^{+}} \frac{ \mathcal{H}^{n} ( E \cap B_{\rho}(x))}{ \omega_{n} \rho^{n} }, \mbox{ for $\mathcal{H}^{n}$-a.e. } x \in \mathbb{R}^{n}. 
\] 
For each  $s \in [0,1],$ we define the \textit{set of points of density $s$ with respect to $E$} as
\[
E^{(s)} := \{ x \in \mathbb{R}^{n} \ : \ \theta (E,x)=s\}
\]
The \textit{essential boundary } $\partial^{e}E$ of  $E$ is defined as the set
\[
\partial^{e} E := \mathbb{R}^{n} \backslash( E^{(0)} \cup E^{(1)} )
\]
\subsection{Approximate limits of measurable functions}
Let $g: \mathbb{R}^{n} \rightarrow \mathbb{R}$ be a Lebesgue measurable function. We define the \textit{approximate upper limit} $g^{\vee}(x)$ and the \textit{approximate lower limit} $g^{\wedge}(x)$ of $g$ at $x \in \mathbb{R}^{n}$ as 
\begin{equation} \label{upper_approx}
g^{\vee}(x) = \inf \left\{ s \in \mathbb{R}: x \in \{ g >s\}^{(0)}\right\} = \inf \left\{ s \in \mathbb{R}: x \in \{ g <s \}^{(1)} \right\}
\end{equation}
and 
\begin{equation} \label{lower_approx}
	g^{\wedge}(x) = \sup \left\{ s \in \mathbb{R}: x \in \{ g <s\}^{(0)}\right\} = \sup \left\{ s \in \mathbb{R}: x \in \{g>s\}^{(1)} \right\},
\end{equation}
respectively. We highlight the fact that both  $g^{\vee}$ and $g^{\wedge}$ are Borel functions and they are defined \textit{for every} $x \in \mathbb{R}^{n}$ with values in $\mathbb{R} \cup \{ \pm \infty\}.$ In addition, if $g_{1}: \mathbb{R}^{n} \rightarrow \mathbb{R}$ and $g_{2}: \mathbb{R}^{n} \rightarrow \mathbb{R}$ are measurable functions such that $g_{1} = g_{2}$ $\mathcal{H}^{n}$-a.e. on $\mathbb{R}^{n}$, then it turns out that 
\[
g_{1}^{\wedge}(x) = g_{2}^{\wedge}(x) \quad \mbox{and} \quad g_{1}^{\vee}(x) = g_{2}^{\vee}(x) \quad \mbox{ for every } x \in \mathbb{R}^{n}.
\]
The \textit{approximate discontinuity set} $S_{g}$ of $g$ is defined as
\[
S_{g} := \{g^{\wedge} \neq g^{\vee} \},
\]
and satisfies $\mathcal{H}^{n} (S_{g}) =0.$ Moreover, even if $g^{\wedge}, \ g^{\vee}$ could take values $\pm \infty$ on $S_{g}$, it turns out that the difference $g^{\vee} - g^{\wedge}$ is well-defined in $\mathbb{R} \cup \{ \pm \infty \} $ for every point $x \in S_{g}.$ In the light of the above considerations, the \textit{approximate jump} $[\, g \,]$ of $g $ is the Borel function  $[\, g \,]:\mathbb{R}^{n} \rightarrow [0,\infty]$ defined as 
\[
[\, g \,](x) := 
\begin{cases}
g^{\vee}(x) - g^{\wedge}(x),& \mbox{ if } x \in S_{g} \\ 
0, & \mbox{elsewhere.}
\end{cases}
\]

Let $E\subset \mathbb{R}^{n}$ be a Lebesgue measurable set. We will say that 
$s \in \mathbb{R} \cup \{ \pm \infty\}$ is the approximate limit of $g$ at $x$ with respect to $E$, denoted by $s = \mbox{aplim} (g , E, x),$ if 
\[
\theta \left( \{ |g-s | > \epsilon \} \cap E ; x \right)=0, \quad \mbox{ for every } \epsilon >0 \qquad (s \in \mathbb{R}),
\]

\[
\theta \left( \{ g< M \} \cap E ; x \right)=0, \quad \mbox{ for every } M >0 \qquad (s= + \infty),
\]
and 
\[
\theta \left( \{ g> - M \} \cap E ; x \right)=0, \quad \mbox{ for every } M >0 \qquad (s = - \infty).
\]

We will say that $x \in S_{g}$ is a \textit{jump point} of $g$ if there exist $\nu \in \partial B_{1}(0)$ such that
\[
g^{\vee}(x)= \mbox{aplim} (g, H_{x,\nu}^{+},x) \quad \mbox{and} \quad 
g^{\wedge}(x)= \mbox{aplim} (g, H_{x,\nu}^{-},x).
\]

In this spirit, we define the \textit{approximate jump direction} $\nu_{g}(x)$ of $g$ at $x$ as $\nu_{g}(x):=\nu$. The set of approximate jump points of $g$ is denoted by $J_{g}$. Note that $J_{g} \subset S_{g}$ and $\nu_{g} : J_{g} \rightarrow \partial B_{1}(0)$ is a Borel function.

\subsection{Functions of Bounded Variation}
Let $\Omega \subset \mathbb{R}^{n}$ be an open set. We denote by $C_{c}^{1}( \Omega ; \mathbb{R}^{n})$ and by $C_{c}( \Omega ; \mathbb{R}^{n})$ the class  of $C^{1}$ functions with compact support and the class of  all continuous functions  with compact support from $\Omega$ to $\mathbb{R}^{n}$, respectively. We  also recall the Sobolev space $W^{1,1} (\Omega )$, that is, the space of all functions $g \in L^{1} (\Omega )$, whose distributional derivative $Dg$ belongs to $L^{1} (\Omega )$.  

Given $g \in L^{1} (\Omega)$, the \textit{total variation} of $g$ in $\Omega $ is defined as 
\[
|Dg| ( \Omega  ) = \sup \left\{ \int_{\Omega } g(x) \ \mbox{div} \ T(x) \ dx : T \in C_{c}^{1} ( \Omega ; \mathbb{R}^{n} ) , \ |T| \leq 1 
\right\}.
\]
We then define the space of functions of bounded variation in $\Omega$, denoted by $BV(\Omega)$, as the set of functions $g \in L^{1} (\Omega)$ such that $|Dg|(\Omega) <\infty$. In addition, we will say that $g \in BV_{loc} (\Omega)$, if $g \in BV(\Omega ')$ for every $\Omega' \subset \subset \Omega$.
If $g \in BV(\Omega )$, due to Radon-Nikodym decomposition of $D g$ with respect to $\mathcal{H}^{n}$, we have
\[
Dg= D^{ac}g + D^{s}g,
\]
where $D^{ac}g$ and $D^{s}g$ are mutually singular measures and $D^{ac}g \ll \mathcal{H}^{n}$.
The density of  $D^{ac}g$ with respect to $\mathcal{H}^{n}$ will be denoted as $\nabla g$, and we have that $\nabla g \in L^{1} (\Omega, \mathbb{R}^{n})$ with $D^{ac}g = \nabla g \  d \mathcal{H}^{n}$. Additionally, it turns out that $\mathcal{H}^{n-1} ( S_{g} \backslash  J_{g} ) = 0$ and $[\, g\, ] \in L^{1}_{loc} ( \mathcal{H}^{n-1} \mres J_{g})$. The \textit{jump part of} $g$ is the $\mathbb{R}^{n}$-valued Radon measure given by
\begin{equation} \label{jump_Radon_measure}
D^{j} g = [\, g\, ]  \nu_{g} d \mathcal{H}^{n-1} \mres J_{g}.
\end{equation}
 Finally, the \textit{Cantorian part} $D^{c}g$ of $Dg$ is defined as the $\mathbb{R}^{n}$-valued Radon measure 
 \[
 D^{c} g = D^{s}g - D^{j} g,
 \]
and is such that $|D^{c}g | (N)=0$ for every set $N \subset \mathbb{R}^{n}$, which is $\sigma$-finite with respect to $\mathcal{H}^{n-1}$.

Note, that in the special case $n=1$, if $(a,b) \subset \mathbb{R}$ is an open interval, every $g \in BV(a,b)$ can be decomposed as the sum
\begin{equation}
    g = g^{ac} + g^{j}+ g^{c},
\end{equation}
where $g^{ac} \in W^{1,1} (a,b)$, $g^{j}$ is a purely jump function  (that is, $D g^{j} = D^{j}g^{j})$ and $g^{c}$ is a purely Cantorian function (that is, $Dg^{c} = D^{c}g^{c}$) (see \cite[Corollary~3.33]{ambrosio2000functions}). Moreover, the total variation $|Dg|$ of $Dg$ can be written as
\begin{equation}\label{1D_totalvariation}
 |Dg| (a,b) = \sup \left\{ \sum_{i=1}^{M} | g(x_{i+1}) - g(x_{i}) |: \ a<x_{1} < x_{2} < \cdots < x_{M}  <b \right\}
\end{equation}
where the supremum is taken over all $M \in \mathbb{N}$ and over all possible partitions of the interval $(a,b)$ with $a< x_{1} < x_{2} < \cdots < x_{M} <b$.

\subsection{Sets of locally finite perimeter in the Euclidean space} \label{setsofFP}

Let $n,m \in \mathbb{N}$ with $1\leq m \leq n$. Let also $E \subset \mathbb{R}^{n}$ be an $\mathcal{H}^{m}$-measurable set. We say that $E$ is a \textit{countably $\mathcal{H}^{m}$-rectifiable} set if there exist a countable family of Lipschitz functions $(g_{j})_{j \in \mathbb{N}}$, where $g_{j}: \mathbb{R}^{m} \rightarrow \mathbb{R}^{n}$,
such that $E \subset_{\mathcal{H}^{m}} \bigcup_{j \in \mathbb{N}} g_{j} (\mathbb{R}^{m})$. In addition, if $\mathcal{H}^{m} ( E \cap K) < \infty$ for every compact set $K \subset \mathbb{R}^{n}$, we say that $E$ is a \textit{locally $\mathcal{H}^{m}$-rectifiable} set.

Let $E \subset \mathbb{R}^{n}$ be a Lebesgue measurable set. We say that $E$ is a set of \textit{locally finite perimeter} in  $\mathbb{R}^{n}$ if there exists an $\mathbb{R}^{n}$-valued Radon measure $\mu_{E}$, such that 
\[
\int_{E} \nabla \psi (x) \ dx = \int_{\mathbb{R}^{n}} \psi (x) \ d \mu_{E}, \quad \mbox{ for every } \psi \in C_{c}^{1} ( \mathbb{R}^{n}).
\]
Note that,  $E$ is a set of locally finite perimeter if and only if  $\chi_{E} \in BV_{loc} (\mathbb{R}^{n})$.
If $G \subset \mathbb{R}^{n}$ is a Borel set, then the relative perimeter of $E$ in $G$ is defined as 
\[
P(E; G) := |\mu_{E} |(G).
\]
When $G=\mathbb{R}^{n}$, we ease the notation to $P(E):=P(E; \mathbb{R}^{n})$.

The \textit{reduced boundary} $\partial^{*}E$ of $E$ is the set of all $x\in \mathbb{R}^{n}$ such that 
\[
\nu_{E}(x) = \lim_{\rho \rightarrow 0^{+}} \frac{ \mu_{E}( B_{\rho}(x))}{ | \mu_{E}|( B_{\rho}(x))} \quad \mbox{ exists and belongs to } \partial B_{1}(0).
\]
The Borel function $\nu_{E}: \partial^{*} E \rightarrow \partial B_{1}(0)$ is usually referred to as the 
 \textit{measure-theoretic outer normal} to $E$. Due to Lebesgue-Besicovitch derivation theorem and \cite[Theorem~3.59]{ambrosio2000functions}, it holds that the reduced boundary $\partial^{*}E$ of $E$ is a locally $(n-1)$-rectifiable set in $\mathbb{R}^{n}$ and 
 \[
 \mu_{E} = \nu_{E} \mathcal{H}^{n-1} \mres \partial^{*}E,
 \]
 so that
\[
\int_{E} \nabla \psi (x) \ dx = \int_{\partial^{*} E} \phi (x) \nu_{E} (x) \ d \mathcal{H}^{n-1} (x) \quad \mbox{ for every } \psi \in C_{c}^{1} ( \mathbb{R}^{n}).
\]
Thus, for every Borel set $G \subset \mathbb{R}^{n}$ we have that 
\[
P (E; G) = | \mu_{E}| (G) = \mathcal{H}^{n-1} (G \cap \partial^{*} E).
\]
Finally, if $E$ is a set of locally finite perimeter, it holds  
\begin{equation} \label{locally_finite_perimeter}
\partial^{*} E \subset E^{(1/2)} \subset \partial^{e}E,
\end{equation}
and additionally, thanks to \textit{Federer's theorem} (see e.g. \cite[Theorem~3.61]{ambrosio2000functions} or \cite[Theorem 16.2]{maggi2012sets}), we have that 
\begin{equation}\label{density_boundary}
\mathcal{H}^{n-1} ( \partial^{e} E \  \backslash \ \partial^{*} E ) =0,
\end{equation}
which implies that the essential boundary $\partial^{e}E$ of $E$ is locally $\mathcal{H}^{n-1}$-rectifiable in $\mathbb{R}^{n}$.

\subsection{Preliminary results}
In this final subsection, we state some results  which will be useful in the following.

The first significant result relates to the set  $E_{z}$ defined in \eqref{slice}. Namely, as it turns out, for $\mathcal{H}^{1}$-a.e $z \in \mathbb{R}$, $E_{z}$ is a set of finite perimeter and  its reduced boundary $\partial^{*} (E_{z})$ enjoys an advantageous property.  These facts follow due to  a variant of a result by Vol'pert \cite{vol1967spaces}, which is provided in \cite[Theorem~2.4]{fusco2013stability}.

\begin{proposition}[Vol'pert]
Let $E$ be a set of finite perimeter in $\mathbb{R}^{n}$. Then for $\mathcal{H}^{1}$-a.e. $z \in \mathbb{R}$ the following hold true:
\begin{enumerate}
    \item[(i)] $E_{z}$ is a set of finite perimeter in $\mathbb{R}^{n-1}$; \vspace{0.2cm}
    \item[(ii)] $\mathcal{H}^{n-2}((\partial^{*} E)_{z} \triangle \partial^{*} (E_{z}))=0$.
\end{enumerate}
\end{proposition}

Thanks to $(ii)$ above, we will often write $\partial^{*} E_{z}$ instead of $(\partial^{*} E)_{z}$ or $\partial (E_{z})$.
The next result presents a crucial regularity property of the function $\ell$, and it can be found in \cite[Lemma~3.1]{fusco2013stability}. 
 \begin{proposition}\label{ell_regular}
Let $E$ be a set of finite perimeter in $\mathbb{R}^{n}$. Then either $\ell(z) =\infty$ for $\mathcal{H}^{1}$-a.e. $z \in \mathbb{R}$, or $\ell (z) < \infty$ for $\mathcal{H}^{1}$-a.e. $ z \in \mathbb{R}$ and $\mathcal{H}^{n} (E) < \infty$. In the latter case, we have $\ell \in BV(\mathbb{R} )$.
\end{proposition}
We present the following auxiliary inequality, which is a special case of \cite[Proposition~3.4]{fusco2013stability}.
\begin{proposition}
Let $\ell :\mathbb{R} \rightarrow [0, \infty)$ be a measurable function, such that $F_{\ell}$ is a set of finite perimeter and finite volume. Let $E \subset \mathbb{R}^{n}$ be an $\ell$-distributed set and let $f: \mathbb{R} \rightarrow [0,\infty]$ be a Borel measurable function. Then
\begin{equation}\label{perimeter_formula}
\int_{\partial^{*}E} f (z) \  d \mathcal{H}^{n-1}(x) \geq \int_{\mathbb{R}} f(z) \sqrt{ (\mathcal{H}^{n-2} (\partial^{*}  E_{z}))^2 + |\nabla \ell (z)|^{2} } \ dz + \int_{\mathbb{R}} f(z) \ d | D^{s} \ell| (z),
\end{equation}
Moreover, if $E = F_{\ell},$ the equality holds in \eqref{perimeter_formula}. 
\end{proposition}

A straightforward consequence of the above result is the following.
\begin{corollary}\label{corollary_perimeter}
Let $\ell: \mathbb{R}\rightarrow [0, \infty)$ be a measurable function, such that $F_{\ell}$ is a set of finite perimeter and finite volume. Then
\begin{equation}\label{corollary_perimeter_formula}
P (F_{\ell}; B \times \mathbb{R}^{n-1}) = \int_{B} \sqrt{ (\mathcal{H}^{n-2} (\partial^{*} (F_{\ell})_{z}))^2 + |\nabla \ell (z)|^{2} } \ dz + | D^{s} \ell| (B), 
\end{equation}
for every Borel set $B \subset \mathbb{R}$.
\end{corollary}

For sake of completeness, we close this preliminary section by presenting the proof of Theorem 1.3 $(ii) \Longrightarrow (i)$.
\begin{proof}[Proof of Theorem 1.3 $(ii) \Longrightarrow (i)$]
Suppose that $(ii)$ holds. Since $\ell \in W^{1,1}(\mathring{J})$, by Proposition \ref{fuscoreg}, the condition \eqref{no flat parts symmetric} is satisfied with $\Omega = \mathring{J}$. In addition, since $J$ is one-dimensional, $\ell$ is absolute continuous in $\mathring{J}$ and therefore,
\[
\ell^{\wedge} (z) = \ell^{\vee}(z)  =\ell (z) >0 \mbox{ for all } z \in \mathring{J}.
\]
Thus, it turns out that \eqref{barchiesi_ell positive} is true. Therefore, due to Theorem 1.2, $(i)$ follows.
\end{proof}

\section{Proof of the Theorem 1.2 $(i) \Longrightarrow (ii)$} \label{results}
We start our analysis with the following lemma, which will be extensively used in the sequel.
\begin{lemma}  \label{densitylemma}
Let $\ell: \mathbb{R}\rightarrow [0, \infty)$ be a measurable function, such that $F_{\ell}$ is a set of finite perimeter and finite volume. Let also $r_{\ell}$ be  defined as in \eqref{aktina} and consider $\bar{z} \in \mathbb{R}.$ Then 
\begin{equation} \label{boundary slice}
( \partial^{*} F_{\ell})_{\bar{z}} =_{\mathcal{H}^{n-1}} B^{n-1} \left( 0 , r_{\ell}^{\vee}(\bar{z})\right) \backslash  B^{n-1} \left( 0 , r_{\ell}^{\wedge}(\bar{z})\right).
\end{equation}
\end{lemma}
\begin{proof}
The proof is divided into two steps. 

\vspace{.3cm}

\noindent
\textbf{Step 1:} We prove
\begin{equation} \label{inclusion1}
( \partial^{*} F_{\ell})_{\bar{z}} \subset \overline{B^{n-1} \left( 0 , r_{\ell}^{\vee}(\bar{z})\right)} \backslash  B^{n-1} \left( 0 , r_{\ell}^{\wedge}(\bar{z})\right).
\end{equation}
To this end, it is enough to prove that

\begin{subequations}
\begin{equation}
    r_{\ell}^{\wedge}(\bar{z})  \leq |w| \quad \mbox{ for every } w \in (\partial^{*} F_{\ell})_{\bar{z}} \label{lower claim_lemma}
\end{equation}
\mbox{and}
\begin{equation}
 r_{\ell}^{\vee}(\bar{z}) \geq |w| \quad \mbox{ for every } w \in (\partial^{*} F_{\ell})_{\bar{z}}. \label{upper claim_lemma}
\end{equation}
\end{subequations}

First, we prove \eqref{lower claim_lemma}. To achieve that, we observe that \eqref{lower claim_lemma} will follow by proving the implication: 

\begin{equation*} \label{density_claim_lemma}
 r_{\ell}^{\wedge}(\bar{z}) >  |w| \Longrightarrow (\bar{z},w)  \in F_{\ell}^{(1)},   
\end{equation*}
or equivalently,
\begin{equation*} \label{density_claim_lemma}
 r_{\ell}^{\wedge}(\bar{z}) >  |w| \Longrightarrow (\bar{z},w)  \in (\mathbb{R}^{n} \backslash F_{\ell})^{(0)}.  
\end{equation*}

To this aim, let  $w \in \mathbb{R}^{n-1}$ be such that $r_{\ell}^{\wedge} (\bar{z}) > |w|$, and let  $\delta >0$ be such that

\[
 r_{\ell}^{\wedge}(\bar{z}) = |w| + \delta.
\]

Let now  $\bar{\rho} \in (0, \frac{\delta}{2}]$. Then,
\[
|w-w'| <  \bar{\rho} \leq  \frac{\delta}{2} \qquad \mbox{ for every } (z',w') \in B_{\bar{\rho}} \left(  (\bar{z},w)) \right).
\]
By virtue of triangle inequality, we have 

\[
r_{\ell}^{\wedge}(\bar{z})  = |w| + \delta \geq |w'| -|w-w'| + \delta > |w'| + \frac{ \delta}{2},
\]
so that 
\begin{equation} \label{latter inequality}
r_{\ell}^{\wedge}(\bar{z}) - \frac{\delta}{2} > |w'| \quad \mbox{ for every } (z',w') \in B_{\bar{\rho}} \left(  (\bar{z},w)) \right).
\end{equation}

\tdplotsetmaincoords{85}{113} 

\begin{figure}[ht]
\centering
\resizebox{0.99\textwidth}{!}
		{

 \begin{tikzpicture}
	    [tdplot_main_coords,
	    cube/.style={very thin,black},
	    grid/.style={very thin,gray },
	    axis/.style={->,black,thick, dotted}]

        \draw[axis] (0,-2,0) -- (0,7,0) node[anchor=west]{\scalebox{0.8}{$z$}}; 
        \draw[axis] (-2,0,0) -- (6,0,0)   node[anchor=east]{\scalebox{0.8}{$w_{1}$}};
        \draw[axis] (0,0,-2) -- (0,0,4) node[anchor=west]{\scalebox{0.8}{$w_{2}$}};

        \begin{scope}[canvas is xz plane at y=1.2]
        
        	\draw[fill=gray!15, opacity=0.7](1.5,0) arc[x radius=1.5, y radius=1.5, start angle=0, end angle=90];
	
	        \draw[fill=gray!15, opacity=0.7] (1.5,0) arc[x radius=1.5, y radius=1.5, start angle=0, end angle=-90];
	        
	        \draw [dotted,fill=gray!15, opacity=0.7] (1.5,0) arc[x radius=1.5, y radius=1.5, start angle=0, end angle=180];
	
	        \draw [dotted,fill=gray!15, opacity=0.7] (1.5,0) arc[x radius=1.5, y radius=1.5, start angle=0, end angle=-180];
        \end{scope}

        
        \draw[-](0,1.2,1.5)-- (0,3.7,2.3);
        
        \draw[-](0,1.1,-1.5)-- (0,3.5,-2.3);

        \begin{scope}[canvas is xz plane at y=3.7]
	        
	        \draw[fill=gray!10, opacity=0.7] (0,0) circle (2.3);
	
        \end{scope}
         
         
            
         
        \begin{scope}[canvas is zx plane at y=3.7]

            \draw[ fill=blue!10, opacity=0.7] (0,0) circle (2.3);

            \node at (1.45,0.9) {\color{blue}\scalebox{0.8}{$r_{\ell}^{\vee}$}\color{black}} ;
            
            \node at (-0.55,-1.5) {\color{red}\scalebox{0.8}{$r_{\ell}^{\wedge}$}\color{black}} ;
            
            \draw (2.295,0)  node[fill,circle,scale=0.25]{} ;
            
            \draw (-0.8,0)  node[fill,circle,scale=0.25]{} ;

        \end{scope}

        \begin{scope}[canvas is xz plane at y=3.7]

	        \draw[fill=white, opacity=0.7 ](0.8,0) arc[x radius=0.8, y radius=0.8, start angle=0, end angle=90];
	        
	        \draw[fill=white, opacity=0.7](0.8,0) arc[x radius=0.8, y radius=0.8, start angle=0, end angle=-90];
	        
	        \draw [dotted,fill=white, opacity=0.7] (0.8,0) arc[x radius=0.8, y radius=0.8, start angle=0, end angle=180];
	    
	        \draw [dotted,fill=white, opacity=0.7](0.8,0) arc[x radius=0.8, y radius=0.8, start angle=0, end angle=-180];
	        
            \node at (0.45,0.25) {\scalebox{0.8}{$\bar{z}$}} ;

        \end{scope}
        
        \begin{scope}[canvas is zx plane at y=3.7]
            
            \draw[dashed,red] (0,0)--(-0.8,0);  
            
            \draw[dashed,blue!85] (0,0)--(2.3,0); 
            
            \draw [decorate,decoration = {brace, mirror,raise=1pt}] (0.8,0) --  (2.3,0);
	        
	        \draw (0.8,0) node[fill,circle,scale=0.25]{};
		    
		    \node at (1.56,-0.75) {\scalebox{0.8}{$\delta$}} ;
	
        \end{scope}

        

        
        \draw[] (0,3.7,0.8)-- (0,5.5,1.2);
        
        \draw[] (0,3.7,-0.8)-- (0,5.4,-1.2);

        \begin{scope}[canvas is xz plane at y=5.5]
	        
	        \draw[fill=gray!20, opacity=0.5]  (0,0) circle (1.2);
	
        \end{scope}	    
         
        \begin{scope}[canvas is xz plane at y=3.7]
         
            \draw (0,-0.8)  node[fill,circle,scale=0.25]{} ;
         
            \draw  node[fill,circle,scale=0.25]{} (0,1);
            
            \node at (0.45,0.25) {\scalebox{0.8}{$\bar{z}$}} ;
            
        \end{scope}  

\end{tikzpicture}

}
 \caption{A graphical illustration of Step 1 for $n=3$.
} 
\label{symmetral fig_label}

\end{figure}

Now, thanks to  \eqref{latter inequality} and the definition of $F_\ell$  \color{black}, we have 
\[
( \mathbb{R}^{n} \backslash F_\ell ) \cap B_{\bar{\rho}} \left(  (\bar{z},w) \right) \subset \left\{ (z',w') \in \mathbb{R} \times \mathbb{R}^{n-1}: r_{\ell}^{\wedge}(\bar{z}) - \frac{\delta}{2} > |w'| \geq r_{\ell}(z') \right\} \cap B_{\bar{\rho}} \left(  (\bar{z},w) \right).
\]

Hence, for every $\rho \in (0, \bar{\rho})$, we have 

\begin{align*}
	\mathcal{H}^{n} ( (\mathbb{R}^{n} \backslash F_\ell) \cap B_{\rho} \left(  (\bar{z},w) ) \right)
	&= \int_{\bar{z} - \rho}^{\bar{z} + \rho} \mathcal{H}^{n-1} ( (\mathbb{R}^{n} \backslash F_\ell ) \cap B_{\rho} ( (\bar{z},w)) \cap \{ z = \zeta\} ) \ d \zeta \\ 
	& \leq \int_{\bar{z} - \rho}^{\bar{z} + \rho} \chi_{ \left\{ r_{\ell} < r_{\ell}^{\wedge}(\bar{z}) - \frac{\delta}{2} \right\}} (\zeta)  \mathcal{H}^{n-1}  ( (\mathbb{R}^{n} \backslash F_\ell )  \cap B_{\rho}((\bar{z},w)) \cap \{z = \zeta\}) \ d \zeta \\
	&= \int_{ ( \bar{z} - \rho, \bar{z} + \rho) \cap \left\{ r_{\ell} < r_{\ell}^{\wedge}(\bar{z}) - \frac{\delta}{2}  \right\}} \mathcal{H}^{n-1} ( (\mathbb{R}^{n} \backslash F_\ell )\cap B_{\rho}(( \bar{z},w)) \cap \{z = \zeta\})  \  d \zeta.
\end{align*}
Now, for $\rho \in (0, \bar{\rho})$, we note that, 
\[
B_{\rho} ( (\bar{z},w)) \cap \{ z = \zeta\} \subset B^{n-1} ( w, \rho) \quad \mbox{ for every } \zeta \in ( \bar{z} - \rho, \bar{z} + \rho).
\]
Therefore, for $\rho \in (0, \bar{\rho})$  we obtain
\begin{align*}
	&\mathcal{H}^{n} ( (\mathbb{R}^{n} \backslash F_\ell ) \cap B_{\rho} ( (\bar{z},w)))  \\ 
	&\leq \int_{ ( \bar{z} - \rho, \bar{z} + \rho) \cap \left\{  r_{\ell} < r_{\ell}^{\wedge}(\bar{z}) - \frac{\delta}{2}  \right\}}  \mathcal{H}^{n-1} (B^{n-1} ( w, \rho )  )\  d \zeta\\
	&= \omega_{n-1} \rho^{n-1}   \int_{ ( \bar{z} - \rho, \bar{z} + \rho) \cap \left\{  r_{\ell} < r_{\ell}^{\wedge}(\bar{z}) - \frac{\delta}{2}  \right\}} 1 \ d \zeta  \\
	&= \omega_{n-1} \rho^{n-1} \  \mathcal{H}^{1} \left(( \bar{z} - \rho, \bar{z} + \rho) \cap \left\{  r_{\ell} < r_{\ell}^{\wedge}(\bar{z}) - \frac{\delta}{2}   \right\} \right).
\end{align*}

Finally, we have
\begin{align*}
	&\lim_{ \rho \rightarrow 0^{+}} \frac{ \mathcal{H}^{n} ( (\mathbb{R}^{n} \backslash F_\ell) \cap B_{\rho}( (\bar{z},w)))}{\omega_{n} \rho^{n} } \\
	&\leq  \frac{ \omega_{n-1}}{\omega_{n}} \lim_{ \rho \rightarrow 0^{+}} \frac{ \mathcal{H}^{1} \left(( \bar{z} - \rho, \bar{z} + \rho) \cap \left\{  r_{\ell} < r_{\ell}^{\wedge}(\bar{z}) - \frac{\delta}{2} \right\} \right)}{ \rho} \\
	&=0,
\end{align*}
where in the last equality we make use of the definition of $r_{\ell}^{\wedge}(\bar{z})$, see \eqref{lower_approx}. This shows \eqref{lower claim_lemma}. Employing an analogous  argument, it can be shown that 

\begin{equation*} \label{density_claim_lemma}
 r_{\ell}^{\vee}(\bar{z}) < |w| \Longrightarrow (\bar{z},w)  \in F_{\ell}^{(0)}.   
\end{equation*}
This implies \eqref{upper claim_lemma}, and finally proves \eqref{inclusion1}.

\vspace{.3cm}

\noindent
\textbf{Step 2:} We conclude the proof. We first observe that by Corollary \ref{corollary_perimeter} with  $B=\{\bar{z}\}$, we obtain
\begin{align*}
\mathcal{H}^{n-1} \left( ( \partial^{*} F_{\ell})_{\bar{z}} \right)
&= \mathcal{H}^{n-1} ( \partial^{*} F_{\ell} \cap \{ z= \bar{z} \} ) \\
&= \mathcal{H}^{n-1} ( \partial^{*} F_{\ell} \cap ( \{ \bar{z} \} \times \mathbb{R}^{n-1}))  \\
&= P (F_{\ell};  \{ \bar{z} \} \times \mathbb{R}^{n-1})\\ 
&= 	\ell^{\vee}(\bar{z}) - \ell^{\wedge}(\bar{z}) \\
&=  \mathcal{H}^{n-1} \left(\overline{ B^{n-1} \left(0, r_{\ell}^{\vee}(\bar{z})\right)} \right)- \mathcal{H}^{n-1} \left( B^{n-1} \left( 0, r_{\ell}^{\wedge}(\bar{z})\right)\right) \\
&= \mathcal{H}^{n-1} \left(\overline{ B^{n-1} \left( 0, r_{\ell}^{\vee}(\bar{z})\right)} \big\backslash 
\left( B^{n-1} \left( ( 0, r_{\ell}^{\wedge}(\bar{z})\right)\right) \right).
\end{align*}

Finally, recalling that, by Step 1

\[
( \partial^{*} F_{\ell})_{\bar{z}} \subset \overline{B^{n-1} \left( 0 , r_{\ell}^{\vee}(\bar{z})\right)} \backslash  B^{n-1} \left( 0 , r_{\ell}^{\wedge}(\bar{z})\right),
\]
we obtain

\begin{align*}
(\partial^{*} F_{\ell} )_{\bar{z}}&=_{\mathcal{H}^{n-1}}  \overline{B^{n-1} \left( 0 , r_{\ell}^{\vee}(\bar{z})\right)} \backslash  B^{n-1} \left( 0 , r_{\ell}^{\wedge}(\bar{z})\right) \\
&=_{\mathcal{H}^{n-1}} B^{n-1} \left( 0, r_{\ell}^{\vee}(\bar{z})\right) \backslash  B^{n-1} \left( 0 , r_{\ell}^{\wedge}(\bar{z})\right),
\end{align*}

which concludes the proof.

\end{proof}

Now, we can  show  that if the set $\{\ell^{\wedge} >0 \}$ fails to be a (possibly unbounded) interval, then rigidity is violated.

\begin{proposition} \label{proposition interval}
Let $\ell: \mathbb{R}\rightarrow [0, \infty)$ be a measurable function,  such that $F_{\ell}$ is a set of finite perimeter and finite volume, and let  $r_{\ell}$ be defined as  in \eqref{aktina}. Suppose that the set $\{ \ell^{\wedge} >0 \}$ is not an interval. That is, suppose that there exists  $\bar{z} \in \{ \ell^{\wedge}=0\} $ such that 
\[
(- \infty, \bar{z} ) \cap \{ \ell^{\wedge} >0 \} \neq \emptyset \quad \mbox{and} \quad (\bar{z}, + \infty) \cap \{ \ell^{\wedge} >0 \} \neq \emptyset.
\]
Then, rigidity is violated. More precisely, setting $E_{1}:= F_{\ell} \cap \{ z < \bar{z} \}$ and $E_{2} = F_{\ell} \backslash E_{1},$ then
\[
E:=E_{1} \cup ( (0, \tau) + E_{2} ) \in \mathcal{K} (\ell) \quad \mbox{ for every } \tau \in \mathbb{R}^{n-1}.
\]
\end{proposition}
\begin{proof}
Let $E_{1}, \ E_{2}$ and $E$ be as in the statement. Let also $\tau \in \mathbb{R}^{n-1}$.
First of all, note that, since $\{z < \bar{z}\}$ is open and $E \cap \{z < \bar{z}\}= F_{\ell} \cap \{z < \bar{z}\}$, we have
\[
E^{(s)} \cap \{z < \bar{z}\} = \left( E \cap \{z < \bar{z}\} \right)^{(s)} = 
\left( F_{\ell} \cap \{z < \bar{z}\} \right)^{(s)}=
F_{\ell}^{(s)} \cap \{z < \bar{z}\}.
\]
for every $s \in[0,1].$ In accordance of that, we infer 
\begin{equation} \label{eq1}
\partial^{*} E \cap \{z < \bar{z}\} = \partial^{*} F_{\ell} \cap \{z < \bar{z}\}.
\end{equation}
In the same fashion, for every $\tau \in \mathbb{R}^{n-1}$, we obtain
\begin{align}
	\partial^{*} E  \cap \{z > \bar{z}\} &= \partial^{*} \left((0, \tau) + F_{\ell} \right)\cap \{z > \bar{z}\} \nonumber \\
	&= \left( (0,\tau) + \partial^{*} F_{\ell} \right) \cap \left( (0,\tau) +  \{z >\bar{z} \} \right) \nonumber \\
	&=(0,\tau)+ \left( \partial^{*} F_{\ell}\cap \{ z > \bar{z}\} \right). \label{eq2}
\end{align}

Hence, due to \eqref{eq1} and \eqref{eq2}, we have
\begin{align*}
P(E) &= \mathcal{H}^{n-1} ( \partial^{*} E \cap \{z < \bar{z}\}) + \mathcal{H}^{n-1} (\partial^{*}E \cap \{z=\bar{z} \}) + \mathcal{H}^{n-1} (\partial^{*} E\cap \{z > \bar{z}\}) \\
&=  \mathcal{H}^{n-1} ( \partial^{*}F_{\ell} \cap \{z < \bar{z}\}) + \mathcal{H}^{n-1} (\partial^{*}E \cap \{z=\bar{z} \}) \\
&+ \mathcal{H}^{n-1} \left( (0,\tau) +  ( \partial^{*} F_{\ell} \cap \{z > \bar{z} \}) \right) \\
&= \mathcal{H}^{n-1} ( \partial^{*} F_{\ell} \cap \{z < \bar{z}\}) + \mathcal{H}^{n-1} (\partial^{*}E \cap \{z=\bar{z} \}) + \mathcal{H}^{n-1} ( \partial^{*} F_{\ell} \cap \{z > \bar{z} \}).
\end{align*}

As a consequence, in order to complete the proof, we need to show that

\begin{equation} \label{main_claim}
\mathcal{H}^{n-1} ( \partial^{*} E \cap \{ z = \bar{z} \}) = \mathcal{H}^{n-1} ( \partial^{*} F_{\ell} \cap \{ z = \bar{z} \}).
\end{equation}

In what will follow, without loss of generality we assume that 
\begin{equation} \label{aplims}
r_{\ell}^{\vee} ( \bar{z}) = \mbox{aplim} (r_{\ell}, ( - \infty, \bar{z}), \bar{z}) \quad \mbox{ and } \quad 
r_{\ell}^{\wedge} (\bar{z}) =\mbox{aplim} (r_{\ell}, (\bar{z}, + \infty), \bar{z})=0.
\end{equation}

We divide the proof of \eqref{main_claim} in several steps.

\vspace{.3cm}

\noindent
\textbf{Step 1: } 

We show that

\begin{equation} \label{claim_step1_no interval}
 ( \partial^{*} E)_{\bar{z} } \subset \overline{B^{n-1} \left( 0 , r_{\ell}^{\vee} (\bar{z})\right)} \cup \{ \tau \}.
\end{equation}

To this end, it suffices to prove that 

\begin{equation}
|w| \leq r_{\ell}^{\vee}( \bar{z} )  \quad \mbox{ for every } w \in  (\partial^{*} E )_{\bar{z}} \backslash \{ \tau \} .
\end{equation}

\vspace{.3cm}

\noindent
\textbf{Step 1a: }  We show that 
\[
 |w| > r_{\ell}^{\vee} (\bar{z}) \Longrightarrow (\bar{z},w) \ \in E_{1}^{(0)}.
\]

To this aim, suppose that there exists $\delta >0$ such that 
\[
|w| = r_{\ell}^{\vee} ( \bar{z}) + \delta.
\]
Then, by arguing as in Step 1 of  Lemma \ref{densitylemma}, for every  $\bar{\rho} \in (0, \frac{\delta}{2})$ we obtain
\[
|w'| > r_{\ell}^{\vee} ( \bar{z})   + \frac{\delta}{2} \quad \mbox{ for every } (z',w') \in B_{\bar{\rho}}((\bar{z},w)).
\]
So, by the definition of $E_{1},$ we have
\begin{align*}
&E_{1} \cap B_{\bar{\rho}}((\bar{z},w)) = F_{\ell} \cap \{ z< \bar{z} \}  \cap  B_{\bar{\rho}}((\bar{z},w)) \\
& \subset \left\{ (z',w') \in \mathbb{R}\times \mathbb{R}^{n-1} : z' < \bar{z} \ \mbox{and}  \ r_{\ell}^{\vee} (\bar{z})
+ \frac{\delta}{2} < |w'| < r_{\ell} (z') \right\} \cap B_{\bar{\rho}}((\bar{z},w)).
\end{align*}

Thus, for every $\rho \in ( 0 , \bar{\rho})$, by similar calculations  as in  Step 1 of Lemma \ref{densitylemma}, we obtain
\begin{align*}
&	\lim_{ \rho \rightarrow 0^{+}} \frac{\mathcal{H}^{n}(E_{1} \cap B_{\rho} ((\bar{z},w)) )}{\omega_{n} \rho^{n}} \\
&\leq \frac{1}{ \omega_{n}} \lim_{ \rho \rightarrow 0^{+}}   \int_{(\bar{z} -\rho, \bar{z}) \cap \{ r_{\ell} > r_{\ell}^{\vee} (\bar{z}) +\frac{\delta}{2}\}} \mathcal{H}^{n-1}  \left( F_{\ell} \cap B_{\rho}((\bar{z},w)) \cap \{ z = \zeta \} \right) \ d \zeta \\
&\leq \frac{ \omega_{n-1}}{\omega_{n}} \lim_{ \rho \rightarrow 0^{+}} \frac{ \mathcal{H}^{1} \left( (  \bar{z} - \rho, \bar{z} ) \cap \left\{  r_{\ell} >  r_{\ell}^{\vee} (\bar{z}) + \frac{\delta}{2} \right\} \right)}{\rho} \\
&=0,
\end{align*}
where in the latter inequality \eqref{aplims} has been used.

\vspace{.3cm}

\noindent
\textbf{Step 1b: }  We show that 

\begin{equation}\label{prop_interval_claim1b}
\{ z= \bar{z} \} \backslash\{(\bar{z},\tau )\}    \subset \left(   (0, \tau )  + E_{2}
\right)^{(0)}.
\end{equation}

To this aim, suppose that $\epsilon:= | w - \tau | >0.$ We will prove that  $(\bar{z},w) \in ( (0,\tau) + E_2)^{(0)}.$ Recalling the argument which was used in the proof of \eqref{latter inequality}, we choose $\bar{\rho}\in (0, \frac{\epsilon}{2})$ such that 
\[
|w' - \tau | > \frac{\epsilon}{2} \quad \mbox{ for every } \quad (z',w') \in B_{\bar{\rho}} ((\bar{z},w)).
\]

Then, we have 
\begin{align*}
&( (0, \tau)  + E_{2} ) \cap B_{\bar{\rho}} ((\bar{z},w)) \\
&= \left(
(0,\tau) + ( F_{\ell} \cap \{  z \geq \bar{z} \})  \right) \cap B_{\bar{\rho}} ((\bar{z},w))\\
&\subset \left\{
(z',w') \in \mathbb{R} \times \mathbb{R}^{n-1} : z' \geq \bar{z}, \ \ \frac{\epsilon}{2} < |w' - \tau | <  r_{\ell} (z')\right\} \cap B_{\bar{\rho}} ((\bar{z},w)).
\end{align*}

We note that, for every $\rho \in (0, \bar{\rho})$, it holds
\[
 B_{\rho} ((\bar{z},w)) \cap \{z  = \zeta \}  \subset B^{n-1} ( w, \rho ) \quad \mbox{ for every } \zeta \in ( \bar{z}, \bar{z} + \rho).
\]

Thus, for every $\rho \in (0 , \bar{\rho}),$

\begin{align*}
\mathcal{H}^{n} \left((0, \tau ) + E_{2} ) \cap B_{\rho} ( ( \bar{z} ,w))\right)
&\leq  \int_{ ( \bar{z} , \bar{z} + \rho) \cap \{ r_{\ell} > \frac{\epsilon}{2} \}} \mathcal{H}^{n-1} (B^{n-1} (  w,  \rho ) )  \  d \zeta  \\
&= \omega_{n-1} \rho^{n-1}   \int_{ ( \bar{z}, \bar{z} + \rho) \cap \{ r_{\ell}  > \frac{\epsilon}{2} \} } 1 \ d \zeta  \\
&= \omega_{n-1} \rho^{n-1} \ \mathcal{H}^{1} \left(( \bar{z}, \bar{z} + \rho) \cap  \left\{ r_{\ell}  > \frac{\epsilon}{2} \right\} \right).
\end{align*}

Based on this,  by \eqref{aplims}  we infer that

\begin{align} \nonumber
\lim_{ \rho \rightarrow 0^{+}} \frac{ \mathcal{H}^{n} \left((0, \tau)  ) + E_{2} ) \cap B_{\rho} ( ( \bar{z} ,w))\right) }{ \omega_{n} \rho^{n}} &\leq \frac{ \omega_{n-1}}{\omega_{n}} \lim_{ \rho \rightarrow 0^{+}} \frac{ \mathcal{H}^{1} \left(( \bar{z}, \bar{z} + \rho) \cap \left\{ r_{\ell} > \frac{\epsilon}{2} \right\} \right)}{\rho} \\ \nonumber
&=0,
\end{align}
which proves \eqref{prop_interval_claim1b}.

\vspace{.3cm}

\noindent
\textbf{Step 1c: } To conclude the proof of Step 1, we observe that, by Step 1a and 1b, as well as by the definition of $E,$ it follows that 

\begin{align*} 
&\bigg\{ (\bar{z},w) \in \mathbb{R} \times \mathbb{R}^{n-1} : |w| >  r_{\ell}^{\vee} (\bar{z})   \bigg\} \bigg\backslash \{ ( \bar{z},\tau ) \} \subset E_{1}^{(0)} \cap  ( (0,\tau) +E_{2} )^{(0)}  = E^{(0)}.
\end{align*}

Therefore, 

\begin{align*}
(\partial^{*}E)_{ \bar{z}} &\subset  \mathbb{R}^{n-1} \bigg\backslash \left( \bigg\{ w \in \mathbb{R}^{n-1} : |w| >  r_{\ell}^{\vee} (\bar{z}) \bigg\} \bigg\backslash\{   \tau\} \right) \\
&= \overline{ B^{n-1} \left( 0,  r_{\ell}^{\vee} (\bar{z})  \right)}\cup \{  \tau\},
\end{align*}
which shows \eqref{claim_step1_no interval}.
\vspace{.3cm}

\noindent
\textbf{Step 2: } Finally, we show \eqref{main_claim}.  Note that, thanks to Step 1, Lemma \ref{densitylemma} and perimeter inequality \eqref{perimeter inequality}, we have 

\begin{align*}
P(E ; \{ z= \bar{z} \} )&= \mathcal{H}^{n-1} ( \partial^{*} E \cap \{z = \bar{z}\}) =  \mathcal{H}^{n-1}  (( \partial^{*} E )_{\bar{z}}) \\
& \leq  \mathcal{H}^{n-1}  \left(
B^{n-1} \left( (0,  r_{\ell}^{\vee} (\bar{z}) \right) \right) \\
&= \mathcal{H}^{n-1}  ( \partial^{*} F_{\ell} \cap \{z  = \bar{z} \})\\
&= P (F_{\ell}; \{z=\bar{z} \}) \leq P(E ; \{z= \bar{z} \}),
\end{align*}

which makes our proof complete.  

\end{proof}

We will now show that, if the jump part $D^{j} \ell$ of $D \ell$ is non-zero, then rigidity is violated.
\begin{proposition} \label{propostion no jumps}
	Let $\ell: \mathbb{R}\rightarrow [0, \infty)$ be a measurable function, such that $F_{\ell}$ is a set of finite perimeter and finite volume, and let  $r_{\ell}$ be  defined as in \eqref{aktina}. Suppose that $\ell$ has a jump at some point $\bar{z} \in \mathbb{R}.$ Then rigidity is violated. More precisely,  
	setting $E_{1}:= F_{\ell} \cap \{ z < \bar{z} \}$ and $E_{2} = F_{\ell} \backslash E_{1},$ then
	\[
	E=:E_{1} \cup ( (0,\tau) + E_{2} ) \in \mathcal{K} (\ell)
	\] 
	
for every  $\tau \in \mathbb{R}^{n-1}$
such that 
\begin{equation} \label{jumpcondition}
0<|\tau | <  r_\ell^{\vee}(\bar{z}) - r_\ell^{\wedge} (\bar{z}).
\end{equation}
\end{proposition}

\begin{proof}
Let $E_{1}, \ E_{2}$ and $E$ be as in the statement. Let also  $\tau \in \mathbb{R}^{n-1}$ be such that \eqref{jumpcondition} is satisfied. It is not restrictive to  assume that 
\begin{equation} \label{jump_approx_limits}
r_{\ell}^{\vee}(\bar{z}) = \mbox{aplim} (r_{\ell}, (- \infty, \bar{z}) , \bar{z}) \quad \mbox{and} \quad r_{\ell}^{\wedge} (\bar{z} ) = \mbox{aplim} (r_{\ell}, ( \bar{z} , + \infty), \bar{z} ).
\end{equation}
By an analogous argument as in the beginning of the proof of  Proposition \ref{proposition interval}, we obtain
\[
P(E)= \mathcal{H}^{n-1} ( \partial^{*} F_{\ell} \cap \{z < \bar{z} \} ) + \mathcal{H}^{n-1} ( \partial^{*} E \cap \{z = \bar{z} \} ) + \mathcal{H}^{n-1} ( \partial^{*} F_{\ell} \cap \{ z > \bar{z} \} ).
\]
Hence, in order to complete the proof, we finally need to show that 
\begin{equation}\label{jump_lemma}
\mathcal{H}^{n-1} (  \partial^{*} E \cap \{ z = \bar{z} \})  = \mathcal{H}^{n-1} (  \partial^{*} F_{\ell} \cap \{ z = \bar{z} \}).
\end{equation}

We divide the proof of \eqref{jump_lemma} into further steps.

\vspace{.3cm}

\noindent
\textbf{Step 1: } We prove that 
\begin{equation} \label{jump_step1_claim}
(\partial^{*} E)_{\bar{z}} \subset \overline{B^{n-1} \left(0, r_{\ell}^{\vee}( \bar{z} )\right)} \backslash B^{n-1} \left( \tau , r_{\ell}^{\wedge}(\bar{z})\right). 
\end{equation}

In order to show \eqref{jump_step1_claim}, it suffices to prove that  
\begin{subequations}
\begin{equation} \label{jump_step1_ineq_B}
 r_{\ell}^{\wedge} (\bar{z})  \leq |w-\tau| \quad \mbox{ for every } w \in (\partial^{*} E)_{\bar{z}}
\end{equation}
and
\begin{equation} \label{jump_step1_ineq_A}
 r_{\ell}^{\vee} (\bar{z}) \geq |w| \quad \mbox{ for every } w \in (\partial^{*} E)_{\bar{z}}.
\end{equation}
\end{subequations}

 First, let us  prove \eqref{jump_step1_ineq_B}. To achieve that, we observe, due to \eqref{locally_finite_perimeter}, our claim will follow if we prove that
\begin{equation} \label{finale}
    |w -\tau| < r_{\ell}^{\wedge} (\bar{z})  \Longrightarrow (\bar{z},w) \in (\mathbb{R}^{n} \backslash E)^{(0)}.
\end{equation}

To this end, suppose  that $ w \in \mathbb{R}^{n-1}$ is such that $|w -\tau| < r_{\ell}^{\wedge} (\bar{z})$. Then, we observe that 
\begin{align*}
\mathbb{R}^{n} \backslash E = ( \mathbb{R}^{n} \backslash E) \cap \{ z < \bar{z} \}) \cup  ( \mathbb{R}^{n} \backslash E) \cap \{ z \geq \bar{z}\} ).
\end{align*}

\begin{figure}[h]
\centering
\resizebox{0.65\textwidth}{!}
{
\subfigure
    {
     \begin{tikzpicture}
    
        \draw[](0.5,-1.75)--(0.5,1.75); 
        %
        
        \draw[](0.5,1.75)--(2.5,1.2);
        \draw[](0.5,-1.75)--(2.5,-1.2);

        \draw[](2.5,0.6)--(2.5,1.2);
        \draw[](2.5,-0.6)--(2.5,-1.2);

      \draw[dotted] (2.5,0) -- (2.5,-2)  ;
      \draw[dotted] (2.5,0) -- (2.5,2);
      \draw[fill=black] (2.5,0) circle (1pt);


      \draw[<->,dotted] (2.9,-0.6) -- (2.9,0) ;
      \draw[dotted] (2.5,-0.6) -- (2.9,-0.6);
      \draw[dotted] (2.5,0) -- (2.9,0);
        \draw[<->,dotted] (2,0) -- (2,1.2) ;
      \draw[dotted] (2,1.2) -- (2.5,1.2);
      \draw[dotted] (2,0) -- (2.5,0);

        \draw[](4,-1)--(4,1);
	    \draw[] (2.5,0.6)-- (4,1);
        \draw[] (2.5,-0.6)-- (4,-1);

        \node at (3.4,-0.3) {\tiny $r_{\ell}^{\wedge}(\bar{z})$} ;
        \node at (1.5,0.6) {\tiny $r_{\ell}^{\vee}(\bar{z})$} ;
        \node at (2,1.8) {$F_{\ell}$};
        \node at (2.25,-1.8) {$\bar{z}$};

\end{tikzpicture}

    }

   \hspace{0.05\textwidth}

\subfigure
    {
       \begin{tikzpicture}

        \draw[](0.5,-1.75)--(0.5,1.75); 
        %
        
        \draw[](0.5,1.75)--(2.5,1.2);
        \draw[](0.5,-1.75)--(2.5,-1.2);

        \draw[](2.5,0.6+0.4)--(2.5,1.2);
        \draw[](2.5,-0.6+0.4)--(2.5,-1.2);
    
      \draw[dotted] (2.5,0) -- (2.5,-2)  ;
      \draw[dotted] (2.5,0) -- (2.5,1.8);
        \draw[fill=black] (2.5,0.4) circle (1pt);

        \draw[](4,-0.6)--(4,1.4);
	    \draw[] (2.5,0.6+0.4)-- (4,1+0.4);
        \draw[] (2.5,-0.6+0.4)-- (4,-1+0.4);

        \node at (1.5,2) {$E_{1}$};
         \node at (3.7,-1) {$(0,\tau)+ E_{2}$};
        \node at (2.2,0.4) {$\tau$};
         \node at (2.25,-1.8) {$\bar{z}$};

\end{tikzpicture}

}
}

\caption{A graphical illustration of Step 1 for $n=2$.}
\label{ell jump counter proof}

\end{figure}
Now,  arguing as  in Step 1 of Proposition \ref{proposition interval}, we infer that
\begin{align} \label{onehand}
\lim_{\rho  \rightarrow 0^{+}} \frac{\mathcal{H}^{n} ( (\mathbb{R}^{n} \backslash E)\cap B_{\rho} (( \bar{z},w) \cap \{z < \bar{z} \})}
{\omega_{n} \rho^{n}} =0.
\end{align}

Hence, to complete the proof of the claim, it remains to show that 
\begin{align} \label{SECONDHAND}
\lim_{\rho  \rightarrow 0^{+}} \frac{\mathcal{H}^{n} ( (\mathbb{R}^{n} \backslash E)\cap B_{\rho} (( \bar{z},w)  \cap \{z \geq \bar{z})\}}
{\omega_{n} \rho^{n}} =0.
\end{align}

 Then there exists $\delta >0$ 
such that 
\[
|w-\tau| + \delta = r_{\ell}^{\wedge}(\bar{z}).
\]

Let now $\bar{\rho} \in (0, \frac{\delta}{2}]$. Then, for each $(z',w') \in B_{\bar{\rho}}((\bar{z},w))$ 
\[
r_{\ell}^{\wedge} (\bar{z}) \geq |w'-\tau| -|w-w'| + \delta > |w'-\tau| - \frac{\delta}{2} +\delta = |w'-\tau| + \frac{\delta}{2},
\]
so that 
\begin{equation}\label{3delta}
r_{\ell}^{\wedge} (\bar{z}) - \frac{\delta}{2} > |w' - \tau| \quad \mbox{ for every } (z',w') \in B_{\bar{\rho}}((\bar{z},w)).
\end{equation}
Then, employing \eqref{3delta} and the definition of the set $E$, we infer 

\begin{align*}
&  (\mathbb{R}^{n} \backslash E)\cap B_{\bar{\rho}}((\bar{z},w)) \cap \{ z \geq \bar{z}\} \\
&\subset \left\{ (z',w') \in \mathbb{R} \times \mathbb{R}^{n-1}: z \geq \bar{z}, \  r_{\ell}^{\wedge}(\bar{z}) - \frac{\delta}{2}>|w'-\tau|\geq r_{\ell} (z') 
\right\} \cap  B_{\bar{\rho}}((\bar{z},w)).
\end{align*}
Moreover, we note that for $\rho \in (0, \bar{\rho})$, we  have 
\[
B_{\rho}((\bar{z},w)) \cap \{ z = \zeta \} \subset B^{n-1} ( w,\rho) \quad \mbox{ for every }  \zeta \in ( \bar{z} - \rho ,\bar{z} + \rho).
\]
As a consequence, for $\rho \in (0, \bar{\rho})$ we obtain 
\begin{align*}
&\mathcal{H}^{n}  \left((\mathbb{R}^{n} \backslash E)\cap B_{\rho}((\bar{z},w)) \cap \{ z \geq \bar{z}\}  \right) \\ 
&\leq \int_{ ( \bar{z} , \bar{z} + \rho) \cap \left\{ r_{\ell}< r_{\ell}^{\wedge}(\bar{z}) - \frac{\delta}{2} \right\}}  \mathcal{H}^{n-1} (B^{n-1} \left( w, \rho ) \right)  \  d \zeta\\
&= \omega_{n-1} \rho^{n-1}   \int_{ ( \bar{z}, \bar{z} + \rho) \cap \left\{ r_{\ell} < r_{\ell}^{\wedge}(\bar{z}) - \frac{\delta}{2} \right\}} 1 \ d \zeta  \\
&= \omega_{n-1} \rho^{n-1}  \mathcal{H}^{1} \left(( \bar{z}, \bar{z} + \rho) \cap \left\{ r_{\ell}  <  r_{\ell}^{\wedge}(\bar{z}) - \frac{\delta}{2}  \right\} \right).
\end{align*}
Then, thanks to \eqref{jump_approx_limits}, we infer

\begin{align*}
&\lim_{\rho \rightarrow 0^{+}} \frac{\mathcal{H}^{n} ( (\mathbb{R}^{n} \backslash E) \cap B_{\rho} ( (\bar{z},w))\cap \{ z \geq \bar{z}\})}{\omega_{n} \rho^{n}} \\
&\leq \frac{\omega_{n-1}}{\omega_{n}} \lim_{\rho \rightarrow 0^{+}}  \frac{  \mathcal{H}^{1} \left(( \bar{z}, \bar{z} + \rho) \cap \left\{ r_{\ell}  < r_{\ell}^{\wedge}(\bar{z}) - \frac{\delta}{2} \right\} \right)}
{\rho} =0,
\end{align*}
where \eqref{lower_approx} has been employed.
This proves \eqref{SECONDHAND}. Then, combining \eqref{onehand} and \eqref{SECONDHAND},  \eqref{finale} follows, and thus  the proof of \eqref{jump_step1_ineq_B} is complete.

Now, for \eqref{jump_step1_ineq_A}, arguing again as in Step 1 of Proposition \ref{proposition interval}, we have 
\begin{equation}
|w|> r_{\ell}^{\vee} (\bar{z}) \Longrightarrow (\bar{z},w) \in E_{1}^{(0)}.
\end{equation}
Making use of similar arguments as above, it can be shown that 
\begin{equation}
|w|> r_{\ell}^{\vee} (\bar{z}) \Longrightarrow (\bar{z},w) \in ((0,\tau)+E_{2})^{(0)},
\end{equation}
which, shows \eqref{jump_step1_ineq_A}, and in turn \eqref{jump_step1_claim}.
\vspace{.3cm}

\noindent
\textbf{Step 2:} We conclude the proof. From \eqref{jumpcondition}, we infer that 
\[
B^{n-1} (\tau, r_{\ell}^{\wedge} (\bar{z}) ) \subset B^{n-1} ( 0, r_{\ell}^{\vee} ( \bar{z})).
\]
As a consequence, thanks to Step 1, Lemma \ref{densitylemma} and perimeter inequality \eqref{perimeter inequality}, we have

\begin{align*}
P(E ; \{ z = \bar{z} \} ) &= \mathcal{H}^{n-1} ( \partial^{*} E \cap \{ z = \bar{z} \} ) = \mathcal{H}^{n-1} ( ( \partial^{*} E)_{\bar{z} }) \\ 
&\leq \mathcal{H}^{n-1} (B^{n-1} ( ( 0, r_{\ell}^{\vee} ( \bar{z})) \backslash B^{n-1} (\tau, r_{\ell}^{\wedge} (\bar{z}) ) ) \\
&= \ell^{\vee}( \bar{z} ) - \ell^{\wedge} ( \bar{z} ) \\
&= P( F_{\ell} ; \{ z= \bar{z} \} ) \\
&\leq P(E ; \{z = \bar{z} \} ).
\end{align*}
From this, we deduce \eqref{jump_lemma}, which completes the proof.
\color{black}
\end{proof}

We are going to prove now that if the Cantorian part  $D^{c} \ell$ of $ D \ell$ is non-zero, then rigidity is violated.

\begin{proposition} \label{proposition_cantorian}
	Let $\ell: \mathbb{R}\rightarrow [0, \infty)$ be a measurable function, such that $F_{\ell}$ is a set of finite perimeter and finite volume. Let also $r_{\ell}$ be as in \eqref{aktina}. Suppose that $D^{c} \ell \neq 0.$  Then rigidity is violated. 
\end{proposition}

\begin{proof}
With no loss of generality, we assume that $\ell$ is a purely Cantorian function. Indeed,  one can decompose $\ell$ as 
\begin{equation} \label{decompositionBV}
\ell = \ell^{a} + \ell^{j} + \ell^{c},
\end{equation}
where $\ell^{a} \in W^{1,1}(\mathbb{R})$, $\ell^{j}$ is purely jump function and $\ell^{c}$ is purely Cantorian. In the case of $\ell^{j} \neq 0$,  the result becomes trivial since, due to  Proposition \ref{propostion no jumps}, rigidity is violated. Furthermore, in the case of $\ell \neq \ell^{c}$, thanks to \eqref{decompositionBV}, the following proof will sustain  only to $\ell^{c}.$ Thus, in what will follow, we assume that
\[
D\ell = D\ell^{c} =D^{c} \ell .
\]

In addition, due to Proposition \ref{proposition interval}, we can assume that  $\{\ell^{\wedge} >0\}$  is an interval, otherwise, the result becomes trivial. Note now that, since $\ell$ is continuous, there exists $a,b >0$ such that $J:= (a,b) \subset \subset \{ \ell^{\wedge}>0\}$ and 
\begin{equation} \label{ell positivity}
\ell(z) >0, \quad \mbox{ for every }  z \in J.
\end{equation}
Since $D^{c} \ell \neq 0,$ we can assume that $| D^{c} \ell | (J) >0.$

We now  fix $\lambda \in (0,1)$, and we define the function $g: \mathbb{R} \rightarrow \mathbb{R}
$ as
\[
g(z)= 
\begin{cases}
0, \quad \quad \quad \quad\quad \quad\quad \   \mbox{ if } z \in (-\infty, a)\\
\lambda( r_{\ell} (z) - r_{\ell}(a)), \quad \mbox{ if } z \in (a, b)\\
\lambda( r_{\ell} (b) - r_{\ell}(a)), \quad \mbox{ if } z \in (b, + \infty).\\
\end{cases}. 
\]

Let us fix a unit vector $e \in \mathbb{R}^{n-1}$. We define the set
\begin{equation} \label{counterset_E}
E:= \left\{ (z,w ) \in \mathbb{R} \times \mathbb{R}^{n-1} : |w- g(z)e| < r_{\ell} (z)  \right\}.
\end{equation}

One can observe that we cannot obtain $E$ using  a single translation on $F_{\ell}$ along $\mathbb{R}^{n-1}$. We are going to prove now that $E \in\mathcal{K} (\ell).$ We divide the proof in several steps.

\vspace{.3cm}

\noindent
\textbf{Step 1:} 
We construct a sequence $\{\ell^{k}\}_{k \in \mathbb{N}}$, where $\ell^{k} : J \rightarrow [0,\infty)$, which satisfies the following properties:

\begin{enumerate}[(i)]
    \item $ r_{\ell}^{k}(z) \longrightarrow r_{\ell}(z)$,  as $k \rightarrow \infty$ for $\mathcal{H}^{1}$-a.e. $z \in J$, \label{i_Cantor_step1}
    \item $D \ell^{k} = D^{j} \ell^{k}$ for every $ k \in \mathbb{N}$, \label{ii_Cantor_step1}
    \item $\displaystyle{\lim_{k \rightarrow \infty} P ( F_{\ell^k} ;  J \times \mathbb{R}^{n-1} ) = P (F_{\ell} ; J \times \mathbb{R}^{n-1})}$. \label{iii_Cantor_step1}
\end{enumerate}

By \eqref{1D_totalvariation} and since $\ell$ is continuous, we have 
\[
|D\ell| (J) = \sup \left\{ \sum_{i=1}^{N-1} | \ell(z_{i+1}) - \ell(z_{i})|: \ a< z_{1} < z_{2}< \cdots < z_{N} <b \right\},
\]
where the supremum runs over $N \in \mathbb{N}$ and over all $z_{1}, z_{2} , \dots, z_{N}$ with $a<z_{1}< z_{2} < \cdots < z_{N}<b$. From that, for every $k \in \mathbb{N}$ there exist $N_{k} \in \mathbb{N}$ and $z_{1}^{k}, \dots, z_{N}^{k}$ with $a< z_{1}^{k} < \cdots < z_{N}^{k} < b$ such that 
\begin{equation} \label{supremum def for ell}
| D \ell| (J)  \leq \sum_{i=1}^{N_{k}-1} | \ell (z_{i+1}^{k}) - \ell (z_{i}^{k})|+ \frac{1}{k}
\end{equation}
and 
\[
| z_{i+1}^{k} - z_{i}^{k}| < \frac{1}{k}, \quad \mbox{ for every } i=1,\dots, N_{k}-1.
\]

It is not restrictive to assume that the partitions are increasing in $k$, i.e. 
\[
\{z_{1}^{k} , \cdots, z_{N_{k}}^{k} \} \subset \{ z_{1}^{k+1} , \cdots, z_{N_{k}+1}^{k+1} \} \quad \mbox{ for every } k \in \mathbb{N}.
\]

We define now for every $k \in \mathbb{N}$
\begin{equation} \label{ell_discrete}
\ell^{k} (z) := \sum_{i=0}^{N_{k}} \ell (z_{i}^{k}) \chi_{[z_{i}^{k}, z_{i+1}^{k})} (z),
\end{equation}
where we set $z_{0}^{k}:=a$ and $z_{N_{k}+1}^{k} :=b.$ 
Moreover, we set 
\[
r_{\ell}^{k}(z):= \left( \frac{\ell^{k} (z)}{\omega_{n-1}} \right)^{\frac{1}{n-1}} \quad \mbox{ for every $z \in J$ and for every $k \in \mathbb{N}$}.
\]
Note that, by definition, $r_{\ell}^{k}=r_{\ell^{k}}$ and $r_{\ell}^{k} \in BV(J).$ By the continuity of $\ell$, we infer that 
\begin{equation} \label{ell_continuity}
\ell^{k} (z) \longrightarrow \ell(z)  \quad  \mbox{ for } \mathcal{H}^{1}\mbox{-a.e. } z \in J.
\end{equation}
Hence, since the map $\eta \longmapsto (\eta/ \omega_{n-1})^{1/n-1}$ is continuous in $(0,\infty )$, we infer that (\ref{i_Cantor_step1}) holds true. Moreover, by \eqref{ell_discrete}, (\ref{ii_Cantor_step1}) holds also true.

For (\ref{iii_Cantor_step1}), we note that thanks to \eqref{ell_continuity}, we have 
\begin{equation} \label{seq_corollary application}
\lim_{k \rightarrow \infty} \mathcal{H}^{n-2} \left( (\partial^{*}F_{\ell^{k}})_{z} \right)  =  \mathcal{H}^{n-2} \left( (\partial^{*}F_{\ell})_{z} \right) \quad \mbox{ for } \mathcal{H}^{1} \mbox{-a.e }  z \in J.
\end{equation}
In addition, 
\begin{align}
|D\ell^{k}| (J) &= \sum_{i=0}^{N_{k}} | \ell (z_{i+1}^{k}) - \ell (z_{i}^{k})| \nonumber  \\
&= | \ell(z_{1}^{k}) - \ell(a) | +   \sum_{i=1}^{N_{k}-1} | \ell (z_{i+1}^{k}) - \ell (z_{i}^{k} )|+| \ell(b) - \ell(z_{N_{k}}^{k})|. \label{Dl^k_bound}
\end{align}
Now, using \eqref{supremum def for ell}, we obtain
\[
|D \ell | (J) - \frac{1}{k} \leq \sum_{i=1}^{N_{k}-1}  | \ell ( z_{i+1}^{k}) - \ell(z_{i}^{k})| \leq |D\ell| (J).
\]
Combining the above inequality with \eqref{Dl^k_bound} and recalling again the continuity of $\ell,$ we infer that
\begin{equation} \label{ell_limit}
|D^{c} \ell|(J)= | D \ell | (J) = \lim _{k \rightarrow \infty} \sum_{i=1}^{N_{k} -1} | \ell (z_{i+1}^{k}) - \ell(z_{i}^{k}) | = \lim_{k \rightarrow \infty} |D \ell^{k} | (J)  =\lim_{k \rightarrow \infty} |D^{c} \ell^{k}|(J).
\end{equation}

Finally, recalling Corollary \ref{corollary_perimeter} and employing \eqref{seq_corollary application}, we obtain
\begin{align*}
\lim_{k \rightarrow \infty} P ( F_{\ell^{k}}; J \times \mathbb{R}^{n-1})& = \lim_{k \rightarrow \infty}\left( \int_{J} \mathcal{H}^{n-2} \left( ( \partial^{*} F_{\ell^{k}})_{z} \right) \ dz + |D^{c} \ell^{k}|(J) \right) \\
&=  \int_{J} \mathcal{H}^{n-2} ( (\partial^{*} F_{\ell})_{z} \ dz + |D^{c} \ell|(J)   \\
&= P (F_\ell ; J \times \mathbb{R}^{n-1}),
\end{align*}
which proves \eqref{iii_Cantor_step1}.
\vspace{.3cm}

\noindent
\textbf{Step 2:} For $k \in \mathbb{N},$ we will construct a $\ell^{k}$-distributed set $E^k$ satisfying 
\[
P (E^{k} ; J \times \mathbb{R}^{n-1}) =  P (F_{\ell^{k}} ; J \times \mathbb{R}^{n-1}).
\]

As a consequence of  \eqref{ii_Cantor_step1} in Step 1, for $k \in \mathbb{N}$ we infer that $D r_{\ell}^{k} = D^{j} r_{\ell}^{k}$ and that the jump set of $r_{\ell^{k}}$ is a finite set. In particular, 
\[
D r_{\ell}^{k} = \sum_{i=1}^{N_{k}} \left( r_{\ell} (z_{i}^{k}) -r_{\ell}( z_{i-1}^{k}) \right) \delta_{z_{i}^{k}},
\]
where, for each $i \in \{1,2, \dots, N_{k} \}$,  $\delta_{z_{i}^{k}}$ denotes the Dirac delta measure concentrated at the point  $z_{i}^{k}$.  Let us now fix $\lambda \in (0,1)$ and we define iteratively the family of sets $\{ E_{i}^{k} \}_{i=1}^{N_{k}} \subset J \times \mathbb{R}^{n-1}$ as 

\begin{align*}
E_{1}^{k}&:= \left[ F_{\ell^{k}} \cap \left( \{ z< z_{1}^{k} \} \backslash \overline{ \{ z < a \} } \right) \right] \cup \left[ \lambda (r_{\ell} (z_{1}^{k}) - r_{\ell} (a) )e + \left( F_{\ell^{k}} \cap \left( \{z < b \} \backslash \{z < z_{1}^{k}\} \right) \right) \right] \\
E_{2}^{k}&:= \left[ E_{1}^{k} \cap  \{ z< z_{2}^{k} \}  \right] \cup \left[ \lambda (r_{\ell} (z_{2}^{k}) - r_{\ell} (z_{1}^{k}) )e + \left( E_{1}^{k}  \backslash \{z < z_{2}^{k}\} \right) \right] \\
& \ \  \ \vdots \\
E_{N_{k}}^{k} &:= \left[ E_{N_{k}-1}^{k} \cap  \{ z< z_{N_{k}}^{k} \}  \right] \cup \left[ \lambda (r_{\ell} (z_{N_{k}}^{k}) - r_{\ell} (z_{N_{k}-1}^{k}) )e + \left( E_{N_{k}-1}^{k}  \backslash \{z < z_{N_{k}}^{k}\} \right) \right]
\end{align*}

Applying Proposition \ref{proposition interval} for each $i \in \{ 1, \dots, N_{k} \},$ we infer that 

\begin{align*}
P(E_{1}^{k} ;  J \times \mathbb{R}^{n-1}) =P(E_{2}^{k} ; J \times \mathbb{R}^{n-1})= \cdots &= P(E_{N_{k}}^{k} ; J \times \mathbb{R}^{n-1})\\
&= P (F_{\ell^{k}} ; J \times \mathbb{R}^{n-1}).
\end{align*}

Note now, that for $i\in \{ 1,2, \cdots, N_{k} \}$ the general term of the above family of sets can be written as 
\begin{align*}
 \label{eqEk}
E_{i}^{k}&= \left[F_{\ell^{k}} \cap  \{ z< z_{1}^{k} \}  \backslash \overline{\{ z< a  \}}  \right]  \cup \left[ \lambda (r_{\ell} (z_{1}^{k}) - r_{\ell} (a) )e + \left( F_{\ell^{k}} \cap \left(  \{ z< z_{2}^{k} \} \backslash \{ z < z_{1}^{k} \} 
\right) \right) \right] \\
& \ \  \cup \left[ \lambda (r_{\ell} (z_{2}^{k}) - r_{\ell} (a))e + \left( F_{\ell^{k}} \cap \left(  \{ z< z_{3}^{k} \} \backslash \{ z < z_{2}^{k} \} 
\right) \right) \right]\\
&\ \  \cup \cdots \cup 
 \left[ \lambda (r_{\ell} (z_{N_{k}}^{k}) - r_{\ell} (a))e + \left( F_{\ell^{k}} \cap \left(  \{ z< b \} \backslash \{ z < z_{N_{k}}^{k} \}
\right) \right) \right].
\end{align*}

Therefore, if we set

\begin{equation} \label{set_Ek}
E^{k}:= E_{N_{k}}^{k} = \left\{ (z,w) \in J \times \mathbb{R}^{n-1} : | w - \lambda (r_{\ell^{k}}  (z) - r_{\ell^{k}} (a))e | < r_{\ell} (z) \right\},
\end{equation}
we conclude that
\[
P(E^{k}; J \times \mathbb{R}^{n-1}) =P (F_{\ell^{k}} ;  J \times \mathbb{R}^{n-1}), \quad \mbox{ for every } k \in \mathbb{N}.
\]

\vspace{.3cm}

	
	\tdplotsetmaincoords{85}{113}
	
	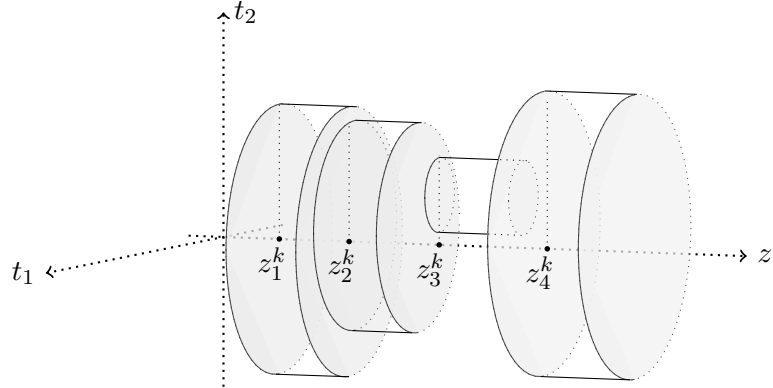
\begin{figure}[h]
		\centering
		\resizebox{0.7\textwidth}{!}
		{
			
			\begin{tikzpicture}
				[tdplot_main_coords,
				cube/.style={very thin,black},
				grid/.style={very thin,gray },
				axis/.style={->,black,thick, dotted}]
				
				
				\draw[axis] (0,-0.5,0) -- (0,7.5,0) node[anchor=west]{$z$};
				\draw[axis] (-2,0,0) -- (6,0,0)   node[anchor=east]{$t_{1}$};
				\draw[axis] (0,0,-2) -- (0,0,3) node[anchor=west]{$t_{2}$};
				
				\begin{scope}[canvas is xz plane at y=0.8]

					\draw [dotted,fill=gray!15, opacity=0.7] (1.5+0.3,0) arc[x radius=1.5+0.3, y radius=1.5+0.3, start angle=0, end angle=180];
					
					\draw [dotted,fill=gray!15, opacity=0.7] (1.5+0.3,0) arc[x radius=1.5+0.3, y radius=1.5+0.3, start angle=0, end angle=-180];

					\draw[fill=gray!15, opacity=0.7](1.5+0.3,0) arc[x radius=1.5+0.3, y radius=1.5+0.3, start angle=0, end angle=90];
					
					\draw[fill=gray!15, opacity=0.7] (1.5+0.3,0) arc[x radius=1.5+0.3, y radius=1.5+0.3, start angle=0, end angle=-90];

					\node at (0.25,-0.3) {$z_{1}^{k}$}; 
					\draw  node[fill,circle,scale=0.2]{} (0,0); 
					\draw[dotted] (0,0)--(0,1.5+0.3); 

				\end{scope}

				\begin{scope}[canvas is xz plane at y=1.8]
					\draw [dotted,fill=gray!15, opacity=0.7] (1.5+0.3,0) arc[x radius=1.5+0.3, y radius=1.5+0.3, start angle=0, end angle=180];
					
					\draw [dotted,fill=gray!15, opacity=0.7] (1.5+0.3,0) arc[x radius=1.5+0.3, y radius=1.5+0.3, start angle=0, end angle=-180];
					
					\draw[fill=gray!15, opacity=0.7](1.5+0.3,0) arc[x radius=1.5+0.3, y radius=1.5+0.3, start angle=0, end angle=90];
					
					\draw[fill=gray!15, opacity=0.7] (1.5+0.3,0) arc[x radius=1.5+0.3, y radius=1.5+0.3, start angle=0, end angle=-90];

				\end{scope}
				

				
				\draw[-](0,0.8,1.504+0.3)-- (0,1.81+0.1,1.504+0.3);
				
				\draw[-](0,0.75,-1.507-0.3)-- (0,1.695+0.1,-1.505-0.3);
				
				

				\begin{scope}[canvas is xz plane at y=1.8 ]
					\draw [dotted,fill=gray!15, opacity=0.7] (1.2,0.2) arc[x radius=1.1+0.3, y radius=1.1+0.3, start angle=0, end angle=180];
					
					\draw [dotted,fill=gray!15, opacity=0.7] (1.2,0.2) arc[x radius=1.1+0.3, y radius=1.1+0.3, start angle=0, end angle=-180];
					
					\draw[fill=gray!15, opacity=0.7](1.2,0.2) arc[x radius=1.1+0.3, y radius=1.1+0.3, start angle=0, end angle=90];
					
					\draw[fill=gray!15, opacity=0.7] (1.2,0.2) arc[x radius=1.1+0.3, y radius=1.1+0.3, start angle=0, end angle=-90];

					\node at (0.25,-0.3) {$z_{2}^{k}$}; 
				
				    \draw  node[fill,circle,scale=0.2]{} (0,0); 
				
				    \draw[dotted] (0,0)--(0,1.5+0.1); 

				\end{scope}
			

				\begin{scope}[canvas is xz plane at y=2.7]
					
					\draw [dotted,fill=gray!15, opacity=0.7] (1.2,0.2) arc[x radius=1.1+0.3, y radius=1.1+0.3, start angle=0, end angle=180];
					
					\draw [dotted,fill=gray!15, opacity=0.7] (1.2,0.2) arc[x radius=1.1+0.3, y radius=1.1+0.3, start angle=0, end angle=-180];
					
					\draw[fill=gray!15, opacity=0.7](1.2,0.2) arc[x radius=1.1+0.3, y radius=1.1+0.3, start angle=0, end angle=90];
					
					\draw[fill=gray!15, opacity=0.7] (1.2,0.2) arc[x radius=1.1+0.3, y radius=1.1+0.3, start angle=0, end angle=-90];

				\end{scope}

				
				\draw[-](2.5,2.95,1.86)-- (2.5,3.72+0.18,1.86);
				\draw[-](2.5,2.91,-0.95)-- (2.5,3.72+0.04,-0.95);
				

				\begin{scope}[canvas is xz plane at y=2.8]
					
					\draw[dotted,fill=gray!10, opacity=0.7](-0.2,0.6) arc[x radius=0.5, y radius=0.5, start angle=0, end angle=180];
					
					\draw[dotted,fill=gray!10, opacity=0.7](-0.2,0.6) arc[x radius=0.5, y radius=0.5, start angle=0, end angle=-180];
					
					\draw[fill=gray!10, opacity=0.7 ](-0.2,0.6) arc[x radius=0.5, y radius=0.5, start angle=0, end angle=90];
					
					\draw[fill=gray!10,opacity=0.7](-0.2,0.6) arc[x radius=0.5, y radius=0.5, start angle=0, end angle=-90];

				\end{scope}   
				

				\begin{scope}[canvas is xz plane at y=3.091]
					
					\draw[dotted] (0,0)--(0,1.19);
					
					\node at (0.3,-.3) {$z_{3}^{k}$};
					
					\draw  node[fill,circle,scale=0.2]{} (0,0);
					
				\end{scope}
				
			    
				\draw[](3.5,4.58,1.5)-- (3.5,5.42,1.5);
				
				\draw[-](3.5,4.58,0.5)-- (3.5,5.262,0.5);

    			%
				
				`       \begin{scope}[canvas is xz plane at y=3.7]
					
					\draw[dotted,fill=gray!10, opacity=0.7](-0.2,0) arc[x radius=0.9+1, y radius=0.9+1, start angle=0, end angle=180];
					\draw[dotted,fill=gray!10, opacity=0.7](-0.2,0) arc[x radius=0.9+1, y radius=0.9+1 , start angle=0, end angle=-180];
					\draw[fill=gray!10, opacity=0.7 ](-0.2,0) arc[x radius=0.9+1, y radius=0.9+1, start angle=0, end angle=90];
					\draw[fill=gray!10, opacity=0.7](-0.2,0) arc[x radius=0.9+1, y radius=0.9+1, start angle=0, end angle=-90];

				\end{scope}

				\begin{scope}[canvas is xz plane at y=4]
					
					\draw[dotted,fill=gray!10, opacity=0.7](-0.2,0.6) arc[x radius=0.5, y radius=0.5, start angle=0, end angle=180];
					
					\draw[dotted,fill=gray!10, opacity=0.7](-0.2,0.6) arc[x radius=0.5, y radius=0.5, start angle=0, end angle=-180];

				\end{scope}  
				
				\draw[dotted](3.5,5.28,0.5)--(3.5,5.74,0.5);
    			\draw[dotted](3.5,5.45,1.5)--(3.5,5.85,1.5);

				
				\begin{scope}[canvas is xz plane at y=5]
					\draw[dotted,fill=gray!10, opacity=0.7](-0.2,0) arc[x radius=0.9+1, y radius=0.9+1, start angle=0, end angle=180];
					\draw[dotted,fill=gray!10, opacity=0.7](-0.2,0) arc[x radius=0.9+1, y radius=0.9+1, start angle=0, end angle=-180];
					\draw[fill=gray!10, opacity=0.7 ](-0.2,0) arc[x radius=0.9+1, y radius=0.9+1, start angle=0, end angle=90];
					\draw[fill=gray!10, opacity=0.7](-0.2,0) arc[x radius=0.9+1, y radius=0.9+1, start angle=0, end angle=-90];
					
				\end{scope}
				
				
					\begin{scope}[canvas is xz plane at y=4.64]
					
					\draw[dotted] (0,0)--(0,2);
					
					\node at (0.25,-0.3) {$z_{4}^{k}$};
					
					\draw  node[fill,circle,scale=0.2]{} (0,0);
					
				\end{scope}

				\draw[-](2.3,5.58,1.1+0.45+0.775)--(2.3,6.9,1.11+0.45+0.765);
				\draw[-](2.3,5.53,-1.11-0.38)--(2.3,6.8,-1.11-0.38);
	
			\end{tikzpicture}

		}
		
		\caption{A graphical illustration of the set $E_{N_{k}}^{k}$ in Step 2.} 
		\label{cantorian step fig_label}
		
	\end{figure}
	

\vspace{.3cm}

\noindent
\textbf{Step 3:} We claim now, that 
\[
E^{k} \longrightarrow \widetilde{E} \quad \mbox{ in }  J \times \mathbb{R}^{n-1}
\]
for some $\ell$-distributed set $\widetilde{E}$ satisfying
\[
P( \widetilde{E} ; J \times \mathbb{R}^{n-1}) = P( F_{\ell}; J \times \mathbb{R}^{n-1}).
\]

Indeed, thanks to \eqref{i_Cantor_step1} of Step 1, it turns out that
\[
r_{\ell}^{k} (z) \longrightarrow r_{\ell} (z) \quad \mbox{ for } \mathcal{H}^{1}\mbox{-a.e. } z \in J.
\]
As a result, recalling \eqref{set_Ek} and  if $\widetilde{E}$ is defined as 
\begin{equation} \label{set_bar_Ek}
\widetilde{E}:=\left\{ (z,w) \in J \times \mathbb{R}^{n-1}: |w - \lambda( r_{\ell}(z) - r_{\ell}(a))e | < r_{\ell}(z) \right\},
\end{equation}
it follows that $\widetilde{E}$ is $\ell$-distributed and $E^{k} \longrightarrow \widetilde{E}$ in $J \times \mathbb{R}^{n-1}.$

Finally, by Step 2, lower semicontinuity of perimeter with respect to $L^{1}$ convergence (see e.g. \cite[Theorem~ 12.15]{maggi2012sets}) and perimeter inequality \eqref{perimeter inequality}, we obtain
\begin{align*}
P(F_{\ell}; J \times \mathbb{R}^{n-1}) &\leq P (\widetilde{E}  ; J \times \mathbb{R}^{n-1}) \leq \liminf_{k \rightarrow \infty} P( E^{k}; J \times \mathbb{R}^{n-1} ) \\
&= \liminf_{k \rightarrow \infty} P( F_{\ell^{k}}; J \times \mathbb{R}^{n-1}) = \lim_{k \rightarrow \infty} P (F_{\ell^{k}}; J \times \mathbb{R}^{n-1})\\
&= P (F_{\ell}; J \times \mathbb{R}^{n-1}), 
\end{align*}

and thus, 
\[
P(\widetilde{E} ; J \times \mathbb{R}^{n-1}) = P( F_{\ell}; J \times \mathbb{R}^{n-1}).
\]

\vspace{.3cm}

\noindent
\textbf{Step 4:} Now consider the set $E$ defined in \eqref{counterset_E}. By previous steps, it turns out that $E$ is $\ell$-distributed, and furthermore

\[
E =_{\mathcal{H}^{n}} \left( F_{\ell} \cap \{ z <a \} \right) \cup \left[ \widetilde{E} \cap \left( \{ z< b \} \backslash \{z <a \} \right)\right] \cup \left[ \lambda ( r_{\ell}(b) - r_{\ell} (a) )e + \left( F_{\ell} \backslash \{ z<b \} \right)
\right].
\]

Since $J \times \mathbb{R}^{n-1} =(a,b) \times \mathbb{R}^{n-1} = \{ z <b \} \backslash \overline{ \{z <a \}}$ and using similar argument as in the proof of Proposition \ref{proposition interval}, we have 

\begin{align*}
P(E)&= P (E ; \{z <a \}) + P(E ; \{z=a\} ) + P(E; \{z<b\} \backslash \overline{ \{z <a \} })\\
&  \quad+ P (E ; \{ z=b\} ) + P( E ; \{z>b\}) \\
&= P (F_{\ell} ; \{ z< a\} ) + P (E ; \{z=a\} )  + P (\widetilde{E} ;\{z<b\} \backslash \overline{ \{z <a \}}) \\
&\quad  + P (E ; \{ z=b\} ) + P( F_{\ell} ; \{z>b\})\\
&= P (F_{\ell} ; \{ z< a\} ) + P (E ; \{z=a\} )  + P (F_{\ell} ;\{z<b\} \backslash \overline{ \{z <a \}}) \\
&  \quad + P (E ; \{ z=b\} ) + P( F_{\ell} ;  \{z>b\}), \\
\end{align*}
where  Step 3 has been employed.

In addition,  an analogous argument as in Step 1 of the proof of Proposition \ref{propostion no jumps} shows that
\[
P ( E; \{ z=a\} ) = P (E; \{z=b\} ) =0.
\]

As a consequence, we infer 

\begin{align*}
P(E)&= P (F_{\ell} ; \{z < a\} ) + P(F_{\ell}; \{z < b\} \backslash  \overline{ \{z<a\} } ) + P( F_{\ell} ;  \{z>b\} )\\
&= P(F_{\ell}).
\end{align*}

Therefore, $E \in K ( \ell )$. In the light of this, the proof is completed.
\end{proof}
We can now show the implication $(i) \Longrightarrow (ii)$ of Theorem \ref{complete characterisation}.
 \begin{proof}[Proof of Theorem 1.2: $(i) \Longrightarrow (ii)$] Assume that $E \in \mathcal{K}(\ell)$. Then, in order to show the result, it is enough to combine  Proposition \ref{proposition interval}, Proposition \ref{propostion no jumps} and Proposition \ref{proposition_cantorian}.
\end{proof}

\section*{Acknowledgments}
The author wishes to  express his gratitude to his PhD advisor  Filippo Cagnetti (University of Sussex), for introducing the problem and for many stimulating discussions and enlighting suggestions on early versions of the manuscript, as well as to  Matteo Perugini (Universit\`a  degli Studi di Milano) for many fruitful discussions. The author was supported by the EPSRC scholarship under the grant EP/R513362/1 \textit{Free-Discontinuity Problems and Perimeter Inequalities under Symmetrisation}.
\bibliographystyle{siam} 
\bibliography{codim_references}
\end{document}